\documentclass[12pt]{amsart}
\usepackage{graphicx}
\usepackage{pinlabel}
\usepackage{color}
\usepackage[T1]{fontenc}
\usepackage[utf8]{inputenc}
\usepackage{tikz}

\theoremstyle{definition} 
\newtheorem{definition}{Definition}[section]

\theoremstyle{plain} 
\newtheorem{proposition}[definition]{Proposition}
\newtheorem{lemma}[definition]{Lemma}
\newtheorem{theorem}[definition]{Theorem} 
\newtheorem*{NNthm}{Theorem} 
\newtheorem{convention}[definition]{Convention}

\theoremstyle{remark} 
\newtheorem{remark}[definition]{Remark}

\newtheorem{example}[definition]{Example}

\newtheorem*{Reader's guide}{Reader 's guide}

\newcommand{\ZZ}{\mathbb{Z}/2\mathbb{Z}}
\newcommand{\ha}[1]{\overset{\rightharpoonup}{{#1}}}

\title[Strata of meromorphic differentials]{Connected components of the strata of the moduli space of meromorphic differentials}
\author{Corentin Boissy}
\address{Aix Marseille Université, CNRS, Centrale Marseille, I2M, UMR 7373, 13453 Marseille France}
\email{corentin.boissy@univ-amu.fr}

\subjclass[2010]{Primary: 32G15. Secondary: 30F30, 57R30}

\keywords{Abelian differentials,  meromorphic differentials, moduli spaces, translation surfaces}

\date{\today}

\begin{document}
\begin{abstract}
We study the translation surfaces corresponding to meromorphic differentials on compact Riemann surfaces. Such geometric structures naturally appear when studying compactifications of the strata of the moduli space of Abelian differentials.

We compute the number of connected components of the strata  of the moduli space of meromorphic differentials. We show that in genus greater than or equal to two, one has up to three components with a similar description as the one of Kontsevich and Zorich for the moduli space of Abelian differentials. In genus one, one can obtain an arbitrarily large number of connected components that are distinghished by a simple topological invariant.
\end{abstract}

\maketitle

\setcounter{tocdepth}{1}
\tableofcontents

%%%%%%%%%%%%%%%%%%%%%%%%%%%%%%%%%%%%%%%%%%
%%%%%%%%%%%%%%%%%%%%%%%%%%%%%%%%%%%%%%%%%%
\section{Introduction}
A nonzero holomorphic one-form ({Abelian differential}) on a compact Riemann surface naturally defines a flat metric with conical singularities on this surface. Geometry and dynamics on such flat surfaces, in relation to geometry and dynamics on the corresponding moduli space of Abelian differentials is a very rich topic and has been widely studied in the last 30 years. It is related to interval exchange transformations, billards in polygons, Teichmüller dynamics. 
%, and Deligne--Munford compactification of moduli spaces.

%A \emph{translation surface} is a topological compact surface that admits, outside a discrete set of points a atlas whose changes of coordinates are translations. It corresponds to an Abelian differential on a compact Riemann surfaces
A noncompact translation surface corresponds to a a one form on a noncompact Riemann surface. The dynamics and geometry on some special cases of noncompact translation surfaces have been studied more recently. For instance, dynamics on $\mathbb{Z}^d$ covers of compact translation surfaces (see \cite{HW,HLT, DHL}), infinite square tiled surfaces (see \cite{Hubert:Schmithuesen}), or general noncompact surfaces (see \cite{Bowman, BV,PSV}).

In this paper, we investigate the case of translation surfaces that come from meromorphic differentials defined on compact Riemann surfaces. In this case, we obtain infinite area surfaces, with ``finite complexity''. 
Dynamics of the geodesic flow on a generic direction on such surface is trivial any infinite orbit converges to the poles. Also, $SL_2(\mathbb{R})$ action doesn't seem as rich as in the Abelian case (see Appendix~A).

 However, it turns out that such structures naturaly appear when studying compactifications of strata of the moduli space of Abelian differentials. Eskin, Kontsevich and Zorich show in a recent paper \cite{EKZ}
 that when a sequence of Abelian differentials $(X_i,\omega_i)$ converges to a boundary point in the Deligne-Munford compactification, then subsets $(Y_{i,j},\omega_{i,j})$ corresponding to thick components of the $X_i$, after suitable rescaling converge to meromorphic differentials (see \cite{EKZ}, Theorem~10). Smillie, in a work to appear, constructs a geometric compactification of the strata of the moduli space of Abelian differentials, by using only flat geometry, and where flat structures defined by meromorphic differentials are needed. 
 
 The connected components of the moduli space of Abelian differentials were described by Kontsevich and Zorich in \cite{KoZo}. They showed that each stratum has up to three connected component, which are described by two invariants: hyperellipticity and the parity of the spin structure, that arise under some conditions on the set of zeroes.
 Later, Lanneau has described the connected components of the moduli space of quadratic differentials. 
 The main goal of the paper is to describe connected components of the moduli space of meromorphic differentials with prescribed poles and zeroes. It is well known that each stratum of the moduli space of genus zero meromorphic differentials is connected. We show that when the genus is greater than, or equal to two, there is an analogous classification as the one of Kontsevich and Zorich, while in genus one, there can be an arbitrarily large number of connected components.

In this paper, we will call \emph{translation surface with poles} a translation surface that comes from a meromorphic differential on a punctured Riemann surface, where poles correspond to the punctured points. We describe in Section~\ref{sec:flat:merom} the local models for neighborhoods of poles.  Similarly to the Abelian case, we  denote by $\mathcal{H}(n_1,\dots ,n_r,-p_1,\dots ,-p_s)$ the moduli space of translation surfaces with poles that corresponds to meromorphic differentials with zeroes of degree $n_1,\dots ,n_r$ and poles of degree $p_1,\dots ,p_s$. 
It will be called \emph{stratum of the moduli space of meromorphic differentials}. We will always assume that $s>0$. A strata is nonempty as soon as $\sum_{i} n_i-\sum_j p_j=2g-2$, for some nonnegative integer $g$ and $\sum_{j}p_j>1$.

For a genus one translation surface $S$ with poles, we describe the connected components by using a geometric invariant easily computable in terms of the flat metric, that we call the \emph{rotation number of a surface}. As we will see in Section~\ref{genus1}, in the stratum $\mathcal{H}(n_1,\dots ,n_r,-p_1,\dots ,-p_s)$, the rotation number is a positive integer that divides all the $n_i,p_j$.

\begin{theorem}\label{MT1}
Let $\mathcal{H}(n_1,\dots ,n_r,-p_1,\dots ,-p_s)$, with $n_i,p_j>0$ and   $\sum_{j}p_j>1$ be a stratum of genus one meromorphic differentials. Denote by $c$ be the number of positive divisors of $\gcd(n_1,\dots ,n_r,p_1,\dots ,p_s)$. The number of connected components of the stratum is:
\begin{itemize}
\item $c-1$ if $r=s=1$.  In this case $n_1=p_1=\gcd(n_1,p_1)$ and each connected component corresponds to a a rotation number that divides $n_1$ and is not $n_1$.
\item $c$ otherwise.  In this case each connected component corresponds to a rotation number that divides $\gcd(n_1,\dots ,n_r, p_1,\dots ,p_s)$.
\end{itemize}
%Furthermore, each rotation number is realized by a unique connected component.
\end{theorem}
A consequence of the previous theorem is that, contrary to the case of Abelian differentials, there can be an arbitrarily large number of connected components for a stratum of meromorphic differentials (in genus~1). For instance, the stratum $\mathcal{H}(24,-24)$ has 7 connected components since the positive divisors of 24 that are not 24 are $1,2,3,4,6,8$ and $12$.

The general classification uses analogous criteria as for Abelian differentials. We recall that in this case, the connected components are distinguished by the following (up to a few exception in low genera):
\begin{itemize}
\item \emph{hyperellipticity}: if there is only one singularity or two singularities of equal degree, there is a component that consists only of hyperelliptic Riemann surfaces. For each translation surface, the hyperelliptic involution is an isometry. Slightly abusing with terminology, we usually call this component the \emph{hyperelliptic component}.
\item \emph{the parity of the spin structure}: If all singularities are of even degree, there are two  connected component (none of which is the hyperelliptic component) distinguished by a topological invariant easily computable in terms of the flat metric.
\end{itemize}

In Section~\ref{invariants}, we define in our context the notion of hyperelliptic component and spin structure.

In the next theorem, we say that the set of poles and zeroes is:
\begin{itemize}
\item of \emph{hyperelliptic type} if the degree of zeroes are or the kind $\{2n\}$ or $\{n,n\}$, for some positive integer $n$, and if the degree of the poles are of the kind $\{-2p\}$ or $\{-p,-p\}$, for some positive integer $p$.
\item of \emph{even type} if the degrees of zeroes are all even, and if the degrees of the poles are either all even, or are $\{-1,-1\}$.
\end{itemize}
%Note in particular that when there are exactly two odd degree poles, the stratum will ``behave'' similarly to the case when all poles are even.

\begin{theorem}\label{MT2}
Let $\mathcal{H}=\mathcal{H}(n_1,\dots ,n_r,-p_1,\dots ,-p_s)$, with $n_i,p_j>0$ be a stratum of genus $g\geq 2$ meromorphic differentials. We have the following.
\begin{enumerate}
\item If $\sum_{i} p_i$ is odd and greater than two, then $\mathcal{H}$ is nonempty and connected.
\item If $\sum_i p_i=2$ and $g=2$, then:
\begin{itemize}
\item if the set of poles and zeroes is of hyperelliptic type, then there are two connected components, one hyperelliptic, the other not (in this case, these two components are also distinghished by the parity of the spin structure)
%the strata $\mathcal{H}(2,2,-2)$, $\mathcal{H}(2,2,-1,-1)$, $\mathcal{H}(4,-2)$ and $\mathcal{H}(4,-1,-1)$ have two connected components, one hyperelliptic, the other not.
\item otherwise, the stratum is connected.
\end{itemize}
\item If $\sum_i p_i>2$ or if $g>2$, then:
\begin{itemize}
\item  if the set of poles and zeroes is of hyperelliptic type, there is exactly one hyperelliptic connected component, and one or two nonhyperelliptic components that are discribed below. Otherwise, there is no hyperelliptic component.
%If $\mathcal{H}$ is in the following list, 
%$$\mathcal{H}(2n,-2p), \mathcal{H}(2n, -p,-p), \mathcal{H}(n,n,-2p), \mathcal{H}(n,n,-p,-p)$$
%for $n,p\geq 1$, then $\mathcal{H}$ contains a hyperelliptic connected component. 
\item if the set of poles and zeroes is of even type, then $\mathcal{H}$ contains exactly two nonhyperelliptic connected components that are distinguished by the parity of the spin structure. Otherwise $\mathcal{H}$ contains exactly one nonhyperelliptic component.
\end{itemize}
\end{enumerate}
\end{theorem}

From the previous theorem, we see that there are at most three connected component in genus greater than or equal to two. For instance, the stratum $\mathcal{H}(4,4,-1,-1)$ contains a hyperelliptic connected component (zeroes and poles are of hyperelliptic type) and two nonhyperelliptic components (the zeroes are even and the poles are $\{-1,-1\}$). So it has three components. The stratum $\mathcal{H}(2,4,-1,-1,-2)$ is connected, since it does not have a hyperelliptic connected component and the poles and zeroes are not of even type. 

The structure of the paper is the following:
\begin{itemize}
\item In Section~\ref{sec:flat:merom}, we describe general facts about the metric defined by a meromorphic differential and define a topology on the moduli space.
\item In Section~\ref{sec:tools}, we present three tools that are needed in the proof. The first two ones appear already in the paper of Kontsevich and Zorich, and in the paper of Lanneau. The third one is a version of the well known Veech construction for the case of translation surfaces with poles.
\item In Section~\ref{genus1}, we describe the connected components in the genus one case. Some of the results in genus one will be very useful for the general genus.
\item In Section \ref{invariants}, we describe the topological invariants for the general genus case, \emph{i.e.} hyperelliptic connected components and the  parity of the spin structure.
\item In Section \ref{sec:minimal}, we compute the connected components for the minimal strata, which are the strata with only one conical singularity (and possibly several poles).
\item In Section \ref{sec:nonminimal}, we compute the connected components for the general case.
\end{itemize}

\subsubsection*{Acknowledgments}
I thank Martin Moeller for many discussions about meromorphic differentials and about spin structures. I thank John Smillie for motivating the work on this paper and interesting discussions. I am also thankful to Pascal Hubert and Erwan Lanneau for the frequent discussions during the development of this paper. This work is partially suported by the ANR Project "GeoDym".

\section{Flat structures defined by meromorphic differentials.}\label{sec:flat:merom}

\subsection{Holomorphic one forms and flat structures}
Let $X$ be a Riemann surface and let $\omega$ be a holomorphic one form. For each $z_0\in X$ such that $\omega(z_0)\neq 0$, integrating $\omega$ in a neighborhood of $z_0$ gives local coordinates whose corresponding transition functions are translations, and therefore $X$ inherits a flat metric, on $X\backslash \Sigma$, where $\Sigma$ is the set of zeroes of $\omega$.

In a neighborhood of an element of $\Sigma$, such metric admits a conical singularity of angle $(k+1)2\pi$, where $k$ is the degree of the corresponding zero of $\omega$. Indeed, a zero of degree $k$ is given locally, in suitable coordinates by $\omega=(k+1)z^k dz$. This form is precisely the preimage of the constant form $dz$ by the ramified covering $z\to z^{k+1}$. In terms of flat metric, it means that the flat metric defined locally by a zero of order $k$ appear as a connected covering of order $k+1$ over a flat disk, ramified at zero.

When $X$ is compact, the pair $(X,\omega)$, seen as a smooth surface with such translation atlas and conical singularities, is usually called a \emph{translation surface}.

If $\omega$ is a meromorphic differential on a compact Riemann $\overline{X}$, we can consider the translation atlas defined defined by $\omega$ on $X=\overline{X}\backslash \Sigma'$, where $\Sigma'$ is the set of poles of $\omega$. We obtain a translation surface with infinite area. We will call such surface \emph{translation surface with poles}. 

\begin{convention}
When speaking of a translation surface with poles $S=(X,\omega)$. The surface $S$ equipped with the flat metric is noncompact. The underlying Riemann surface $X$ is a punctured surface and $\omega$ is a holomorphic one-form on $X$. The corresponding closed Riemann surface is denoted by $\overline{X}$, and $\omega$ extends to a meromorphic differential on $\overline{X}$ whose set of poles is precisely $\overline{X}\backslash X$.
\end{convention}

Similarly to the case of Abelian differentials. A \emph{saddle connection} is a geodesic segment that joins two conical singularities (or a conical singularity to itself) with no conical singularities on its interior. 

We also recall that it is well known that $\sum_{i=1}^r n_i-\sum_{j=1}^s p_j=2g-2$, where $\{n_1,\dots ,n_r\}$ is the set (with multiplicities) of degree of zeroes of $\omega$ and $\{p_1,\dots ,p_s\}$ is the set (with multiplicities) of degree of the poles of $\omega$. 

\subsection{Local model for poles}
 The neighborhood of a pole of order one is an infinite cylinder with one end. Indeed, up to rescaling, the pole is given in local coordinates by $\omega=\frac{1}{z}dz$. Writing $z=e^{z'}$, we have $\omega=dz'$, and $z'$ is in a infinite cylinder.

%\subsection{Local models for higher degree poles}
Now we describe the flat metric in a neighborhood of a pole of order $k\geq 2$ (see also \cite{strebel}).
First, consider the meromorphic 1-form on $\mathbb{C}\cup \{\infty \}$ defined on $\mathbb{C}$ by $\omega=z^k dz$. Changing coordinates $w=1/z$, we see that this form has a pole $P$ of order $k+2$ at $\infty $, with zero residue. In terms of translation structure, a neighborhood of the pole is obtained by taking an infinite cone of angle $(k+1)2\pi$ and removing a compact neighborhood of the conical singularity. Since the residue is the only local invariant for a pole of order k, this gives a local model for a pole with zero residue.

Now, define $U_R=\{z\in \mathbb{C}| |z|>R\}$ equipped with the standard flat metric.
Let $V_R$  be the Riemann surface obtained after removing from $U_R$ the ${\pi}$--neighborhood of the real half line $\mathbb{R}^-$, and identifying by the translation $z\to z+\imath 2\pi$ the lines $-\imath {\pi}+\mathbb{R}^-$ and $\imath {\pi}+\mathbb{R}^-$. The surface $V_R$ is naturally equipped with a holomorphic one form  $\omega$ coming from $dz$ on $V_R$. We claim that this one form has a pole of order 2 at infinity and residue -1.
Indeed, start from the one form on $U_{R'}$ defined by $(1+1/z)dz$ and integrate it. Choosing the usual determination of $\ln(z)$ on $\mathbb{C}\backslash \mathbb{R}^-$, 
one gets the map $z\to z+\ln(z)$ from $U_{R'}\backslash \mathbb{R}^-$ to $\mathbb{C}$, which extends to a injective holomorphic map $f$ from $U_{R'}$ to $V_R$, if $R'$ is large enough. Furthermore, the pullback of the form $\omega$ on $V_R$ gives $(1+1/z)dz$. Then, the claim follows easily after the change of coordinate $w=1/z$

Let $k\geq 2$. The pullback of the form $(1+1/z)dz$ by the map $z\to z^{k-1}$ gives $((k-1)z^{k-2}+(k-1)/z)dz$, \emph{i.e.} we get at infinity a pole of order $k$ with residue $-(k-1)$. In terms of flat metric, a neighborhood of a pole of order $k$ and residue $-(k-1)$ is just the natural cyclic $(k-1)$--covering of $V_R$. Then, suitable rotating and rescaling gives the local model for a pole of order $k$ with a nonzero residue.

For flat geometry, it will be convenient to forget the term $2\imath \pi$ when speaking of residue, hence we define the \emph{flat residue} of a pole $P$ to be $\int_{\gamma_P} \omega$, where $\gamma_P$ is a small closed path that turns around a pole counterclokwise.

\subsection{Moduli space}

If $(X,\omega)$ and $(X',\omega')$ are such that there is a biholomorphism $f:X\to X'$ with $f^* \omega'=\omega$, then $f$ is an isometry for the metrics defined by $\omega$ and $\omega'$. Even more, for the local coordinates defined by $\omega,\omega'$, the map $f$ is in fact a translation. 

As in the case of Abelian differentials, we consider the moduli space of meromorphic differentials, where $(X,\omega)\sim (X',\omega')$ if there is a biholomorphism $f:X\to X'$ such that $f^* \omega'=\omega$.  A stratum corresponds to prescribed degree of zeroes and poles. We denote by $\mathcal{H}(n_1,\dots ,n_r,-p_1,\dots ,-p_s)$ the \emph{stratum} that corresponds to meromorphic differentials with zeroes of degree $n_1,\dots ,n_r$ and poles of degree $p_1,\dots ,p_s$. Such stratum is nonempty if and only if $\sum_{i=1}^r n_i-\sum_{j=1}^s p_j=2g-2$ for some integer $g\geq 0$ and if $\sum_{j=1}^s p_j>1$ (\emph{i.e.} if there is not just one simple pole.). A \emph{minimal stratum} is a stratum with $r=1$, \emph{i.e.} which corresponds to surfaces with only one conical singularity and possibly several poles.

We define the topology on this space in the following way: a small neighborhood of $S$, with conical singularities $\Sigma$, is defined to be the equivalent classes of surfaces $S'$ for which there is a differentiable injective map $f:S\backslash V(\Sigma)\to S'$ such that $V(\Sigma)$ is a (small) neighborhood of $\Sigma$, $Df$ is close the identity in the translation charts, and the complement of the image of $f$ is a union on disks. One can easily check that this topology is Hausdorff.

%Local coordinates for a stratum are obtained  by integrating the one form $\omega$ along a basis of $H_1(X,zeroes(\omega);\mathbb{Z})$ (see for instance Section~\ref{rq:param}). Recall that $X$ is the punctured surface obtained after removing the poles.

%\comment{TURN AROUND THE POLES????}

\section{Tools}\label{sec:tools}

\subsection{Breaking up a singularity: local construction}\label{bzero:loc}
Here we describe a sur\-ge\-ry, introduced by Eskin, Masur and Zorich (see \cite{EMZ}, Section~8.1) for Abelian differentials, that ``break up'' a singularity of degree $k_1+k_2\geq 2$ into two singularities of degree $k_1\geq 1$ and $k_2\geq 1$ respectively. This surgery is local, since the metric is modified only in a neighborhood of the singularity of degree $k_1+k_2$. In particular, it is also valid for the flat structure defined by a meromorphic differential.

\begin{figure}[htb]
\begin{center}
\includegraphics[width=360pt]{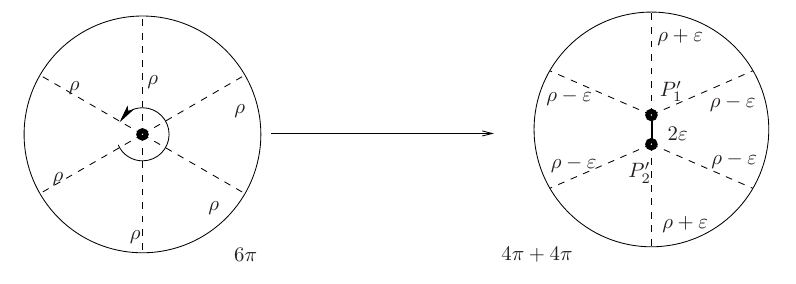}
\caption{Breaking up a zero, after Eskin, Masur and Zorich}
\label{bzero}
\end{center}
\end{figure}

We start from a singularity of degree $k_1+k_2$. A neighborhood of such singularity is obtained by gluing $(2k_1+2k_2+2)$ Euclidean half disks in a cyclic order. The singularity breaking procedure consists in changing continuously the way these half disks are glued together, as in Figure~\ref{bzero}. This breaks the singularity of degree $k_1+k_2$ into  singularities of degree $k_1$ and $k_2$ respectively, and with a small saddle connection joining them.

\subsection{Bubbling a handle}\label{bubbling}
The following surgery was introduced by Kontsevich and Zorich in \cite{KoZo}. Since it is a local construction, it is also valid for meromorphic differentials. 
As before, we start from  a singularity of degree $k_1+k_2$ on a surface $S$. We first apply the previous surgery to get a pair of singularities of degree $k_1$ and $k_2$ respectively, and with a small saddle connection $\gamma$ joining them. Then, we cut the surface along $\gamma$ and obtain a flat surface with a boundary component that consists of two geodesic segments $\gamma_1,\gamma_2$. We identify their endpoints and the corresponding segments are now closed boundary geodescis $\gamma_1',\gamma_2'$. Then, we consider a cylinder with two boundary components isometric to $\gamma_i'$, and glue each of these component to $\gamma_i'$. The angle between $\gamma_1'$ and $\gamma_2'$ is $(k_1+1)2\pi$ (and $(k_2+1)2\pi$) 

Using a notation similar to the one introduced by Lanneau in \cite{La:cc}, we will denote by $S\oplus (k_1+1)$ the resulting surface for an arbitrary choice of continuous parameters. Different choices of continuous parameters lead to the same connected component and from a path $(S_t)_{t\in [0,1]}$, one can easily deduce a continuous path $S_t\oplus (k_1+1)$. Hence, as in \cite{La:cc}, the connected component of $S\oplus s$ only depends on $s$ and on the connected component of $S$. So, if $S$ is in a connected component $\mathcal{C}$ of a stratum of Abelian (resp. meromorphic) differential with only one singularity, $\mathcal{C}\oplus s$ is the connected component of a stratum of Abelian (resp. meromorphic) differentials obtained by the construction. 

\begin{remark}
{The notation $\oplus$ slightly differs to the one introduced by Lanneau: since he manipulates \emph{quadratic differentials}, the angles can be any multiples of $\pi$, while in our case, we only have multiples of $2\pi$. So the surface we obtain would have been written $S\oplus 2(k_1+1)$ with the notation of Lanneau.}
\end{remark}

The following lemma is Proposition~2.9 in the paper of Lanneau \cite{La:cc}, written in our context. The ideas behind this proposition were also in the paper of Kontsevich and Zorich \cite{KoZo}. 
\begin{lemma}\label{lemme:lanneau}
Let $\mathcal{C}$ be connected component a minimal stratum  of the moduli space of meromorphic differentials, of the form $\mathcal{H}(n,-p_1,\dots ,-p_r)$. Then, the following statements hold.
\begin{enumerate}
\item $ \mathcal{C}\oplus s_1 \oplus s_2=\mathcal{C}\oplus s_2 \oplus s_1$ if $1\leq s_1,s_2\leq n+1$ and either $s_1\neq \frac{n}{2}+1$ or $s_2 \neq \frac{n+2}{2}+1$.
\item $ \mathcal{C}\oplus s_1 \oplus s_2=\mathcal{C}\oplus s_2-1 \oplus s_1+1$ if $1\leq s_1\leq n+1$ and $2\leq s_2\leq n+3$.
\item $\mathcal{C}\oplus s_1 \oplus s_2=\mathcal{C}\oplus s_2-2 \oplus s_1$ if $1\leq s_1\leq n+1$ and $1\leq s_2\leq n+3$ and  $s_2-s_1\geq 2$.
\item $\mathcal{C}\oplus s=\mathcal{C}\oplus (n+2-s)$ for all $s\in \{1,\dots, n+1\}$
\end{enumerate}
\end{lemma}
\begin{remark}
There is a small mistake in the statement of Lanneau: the condition ``either $s_1\neq \frac{n}{2}+1$ or $s_2 \neq \frac{n+2}{2}+1$'' does not appear while it is necessary.

This leads to a gap in the proof of Lanneau's  Lemma~6.13 in \cite{La:cc}, but this gap is easily solved by using  Lemma~A.2 of the same paper.
\end{remark}

\subsection{The infinite zippered rectangle construction}
In this section, we describe a construction of translation surfaces with poles which is analogous to the well known Veech zippered rectangle construction. We will call this construction the \emph{infinite zippered rectangles construction}. %Since the formal descriptio

We first recall Veech's construction.

\subsubsection{The Veech construction of a translation surface}\label{veech:constr}
The {Veech construction}, or {zippered rectangle construction} is usually seen as a way to define a suspension over an interval exchange map (see \cite{Veech82}). We can also see it as a easy way to define (almost any) translation surface. 
Consider a finite alphabet $\mathcal{A}=\{\alpha_1,\dots ,\alpha_d\}$, and a pair on one to one maps $\pi_t,\pi_b:\mathcal{A}\to \{1,\dots ,d\}$.  Let 
$\zeta\in \mathbb{C}^\mathcal{A}$ be a vector for which each entry has positive real part.

The Veech construction can be seen in two (almost) equivalent ways. One with a $2d$ sided polygon, and one with $d$ rectangles that are identified on their boundary. 

We present the first one, which is simpler but not as general as the second one. 
Consider the broken line $L_t$ on $\mathbb{C}=\mathbb R^2$ defined by concatenation of the vectors $\zeta_{\pi_t^{-1}(j)}$ (in this order) for $j=1,\dots,d$ with starting point at the origin. Similarly, we consider the broken line $L_b$  defined by concatenation of the vectors $\zeta_{\pi_b^{-1}(j)}$ (in this order) for $j=1,\dots,d$ with starting point at the origin.

We assume that $\zeta$ is such that the vertices of $L_t$ are always above the real line, except possibly the foremost right (and of course the one at the origin), and that similarly, the vertices of $L_b$ are below the real line. Such $\zeta$ is called \emph{suspension datum} (see \cite{MMY}), and exists under a combinatorial condition on  $(\pi_t,\pi_b)$ usually called ``irreducibility''.

If the lines $L_t$ and $L_b$ have no intersections other than the endpoints, we can construct a translation surface $X$ by identifying each side $\zeta_j$ on $L_t$ with the side $\zeta_j$ on $L_b$ by a translation. The resulting surface is a translation surface endowed with  the form $\omega=dz$ (see Figure~\ref{Veech:construction}).

%%% TIKZ TIKZ TIKZ TIKZ TIKZ TIKZ TIKZ TIKZ TIKZ TIKZ TIKZ TIKZ
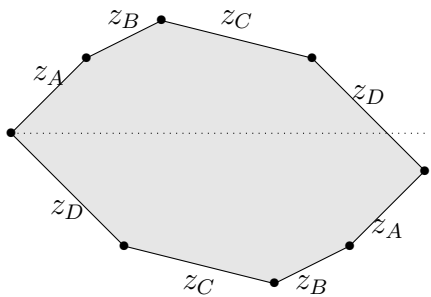
\begin{figure}[htb]
\begin{tikzpicture}
\coordinate (z1) at (1,1);
\coordinate (zz1) at (-1,-1);
\coordinate (z2) at (1,0.5);
\coordinate (zz2) at (-1,-0.5);
\coordinate (z3) at (2,-0.5) ;
\coordinate (zz3) at (-2,0.5) ;
\coordinate (z4) at (1.5,-1.5) ;
\coordinate (zz4) at (-1.5,1.5) ;
\draw [black,fill=gray!20] (0,0)--++ (z1) node[midway,above] {$z_A$} node {\tiny $\bullet$}
                     --++ (z2) node[midway,above] {$z_B$} node {\tiny $\bullet$}
                     --++ (z3) node[midway,above] {$z_C$} node {\tiny $\bullet$}
                     --++ (z4) node[midway,above] {$z_D$} node {\tiny $\bullet$}
                     --++ (zz1) node[midway,below] {$z_A$} node {\tiny $\bullet$}
                     --++ (zz2) node[midway,below] {$z_B$} node {\tiny $\bullet$}
                     --++(zz3) node[midway,below] {$z_C$} node {\tiny $\bullet$}
                     --++(zz4) node[midway,below] {$z_D$}  node {\tiny $\bullet$}                  
                     --cycle;
\draw[dotted] (0,0)--(5.5,0);
\end{tikzpicture}
\caption{Veech's construction of a translation surface}
\label{Veech:construction}
\end{figure}
% END TIKZ    END TIKZ    END TIKZ    END TIKZ    END TIKZ    END TIKZ

\begin{remark}
The surface constructed in this way can also be seen as a union of rectangles whose dimensions are easily deduced from $\pi_t,\pi_b$ and $\zeta$, and that are ``zippered'' on their boundary. One can define $S$ directely in this way: the construction works also if $L_t,L_b$ have other intersection points.
This is the  \emph{zippered rectangle construction}, due to  Veech (\cite{Veech82}). This construction coincides with the first one in the previous case.
\end{remark}

\subsubsection{Basic domains}\label{basic:domains}
Now we generalize the previous construction to obtain a flat surface that corresponds to a compact Riemann surface with a meromorphic differential. Instead of having a polygon with pairwise identification on its boundary, we will have a finite union of some   ``basic domains'' which are half-planes and infinite cylinders with polygonal boundaries (see Figure~\ref{fig:basic:domains}).

Let $n\geq 0$. Let $\zeta\in \mathbb{C}^n$ be a complex vector whose entries have positive real part. 

Consider the broken line $L$ on $\mathbb{C}$ defined by concatenation of the following:
\begin{itemize}
\item the half-line $l_1$ that corresponds to $\mathbb{R}^-$,
\item the broken line $L_t$ defined as above, \emph{i.e.}  the concatenation of the segment defined by the vectors $\zeta_{j}$ (in this order) for $j=1,\dots,n$ with starting point at the origin,
\item the horizontal half line $l_2$ starting from the right end of $L_t$, and going to the right.
\end{itemize}
We consider the subset $D^+(z_1,\dots ,z_n)$ (resp. $D^-(z_1,\dots ,z_n)$) as the set of complex numbers that are above $L$. The line $l_1$ will be refered to as the \emph{left} half-line, and $l_2$ will be refered to as the \emph{right} half-line. We will sometime write such domains $D^{+}$ or $D^-$ for short.
The sets $D^\pm$ are kinds of half-planes with polygonal boundaries.
Note that $n$ might be equal to $0$, and in this case, $D^+$ (resp. $D^-$) is just a half-plane with a marked point on its (horizontal) boundary.

Similarly, if $n\geq 1$, we can define the subset $C^+(z_1,\dots ,z_n)$ (resp. $C^-(z_1,\dots ,z_n)$) as the set of complex numbers that are above $L_t$. Its boundary consists of two infinite vertical half-lines joined by the broken line $L_t$. The two infinite half-lines will be identified in the resulting construction, hence $C^\pm$ is just an infinite half-cylinder with a polygonal boundary.

%%% TIKZ TIKZ TIKZ TIKZ TIKZ TIKZ TIKZ TIKZ TIKZ TIKZ TIKZ TIKZ
%%%%%%%%% D^+    et     C^+
\begin{figure}[htb]
\label{ex:D:C}
\begin{tikzpicture}[scale=0.5]
\coordinate (z1) at (1,0.5);
\coordinate (z2) at (1,-1.5);
\coordinate (z3) at (2,2.5) ;
\coordinate (z4) at (1.5,-1) ;
 \draw [dashed] (-1.5,0)--(0,0);
\draw [dashed] (11,1.5)--(12.5,1.5);
 \draw [dashed] (16,7.5)--(16,5.5);
\draw [dashed] (20,7.5)--(20,5.5);

\fill [fill=gray!20] (-1.5,0)--(0,0)--++(3.5,0) --++ (z1) --++ (z2) --++(z3) --++(5,0) --++(0,4.5)--++(-14,0) node[midway,below] {$D^+$} --cycle;

\draw (0,0)--++(3.5,0) node[midway,above] {$l_1$} node {\tiny $\bullet$}
                     --++ (z1) node[midway,above] {$\zeta_1$} node {\tiny $\bullet$}
                     --++ (z2) node[midway,above right] {$\zeta_2$} node {\tiny $\bullet$}
                     --++(z3) node[midway,below right] {$\zeta_3$} node {\tiny $\bullet$}
                     --++(3.5,0) node[midway, above] {$l_2$};

\fill [fill=gray!20] (16,0) --++ (z1) --++ (z2) --++(z3) --++(0,6)
                    --++(-4,0)  node[midway,below] {$C^+$}--cycle;

\draw (16,5.5) --++(0,-5.5)
                     --++ (z1) node[midway,above] {$\zeta_1$} node {\tiny $\bullet$}
                     --++ (z2) node[midway,above right] {$\zeta_2$} node {\tiny $\bullet$}
                     --++(z3) node[midway,below right] {$\zeta_3$} node {\tiny $\bullet$} 
                     --++ (0,4);
\end{tikzpicture}
\caption{A domain $D^+(\zeta_1,\zeta_2,\zeta_3)$ and a domain $C^+(\zeta_1,\zeta_2,\zeta_3)$}
\label{fig:basic:domains}
\end{figure}
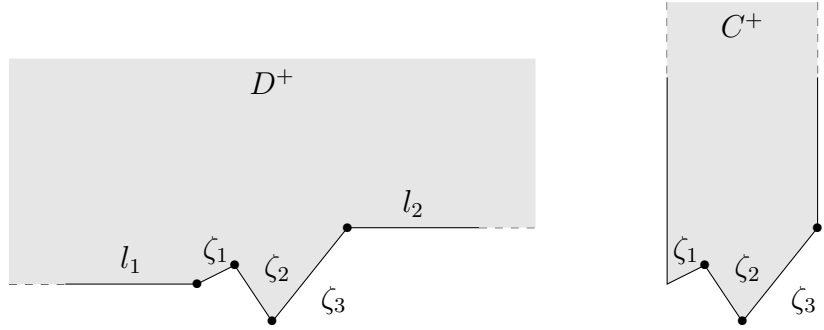
%%% END  TIKZ     END TIKZ     END TIKZ     END TIKZ 

\subsubsection{An example: a surface with a single pole of order 2.}
The idea of the construction is to glue by translation the basic domains together in order to get a noncompact translation surface with no boundary components. Since the formal description is rather technical, we first present a simple version of the construction. 

Let $\mathcal{A}$ be a finite alphabet and $\pi_t,\pi_b:\mathcal{A}\to \{1,\dots ,d\}$ be one-to-one maps.
Let $\zeta\in \mathbb{C}^\mathcal{A}$ be such that $Re(\zeta_\alpha)>0$ for all $\alpha\in \mathcal{A}$.

We define a flat surface $S$ as the disjoint union of the two half-planes  $D^+=D^+(\zeta_{\pi_t^{-1}(1)},\dots ,\zeta_{\pi_t^{-1}(n)})$ and $D^-=D^-(\zeta_{\pi_b^{-1}(1)},\dots ,\zeta_{\pi_b^{-1}(n)})$ glued on their boundary by translation: the left half line of $D^+$ being glued to the left half-line of $D^-$ and similarly with the right half-lines, and a segment in $D^+$ corresponding to some $\zeta_i$ is glued to the corresponding one of $D^-$.

Note that contrary to the case of compact translation surfaces, there is no ``suspension data condition'' on $\zeta$, hence, no combinatorial condition on $\pi$. The only condition that we require is that $Re(\zeta_i)>0$ for all $i$. Note also that we can have $n=0$, in this case $S=\mathbb{C}$.

%%% TIKZ TIKZ TIKZ TIKZ TIKZ TIKZ TIKZ TIKZ TIKZ TIKZ TIKZ TIKZ
%%%%%%%%% S in H(-2,...)
\begin{figure}[htb]
\label{ex:surf:av:pole2}
\begin{tikzpicture}[scale=0.6]
\coordinate (z1) at (1,0.5);
\coordinate (z2) at (1,-1.5);
\coordinate (z3) at (2,2.5) ;
\coordinate (z4) at (1.5,-1) ;
 \draw [dashed] (-5,0)--(-3.5,0);
\draw [dashed] (9,0.5)--(10.5,0.5);
\fill [fill=gray!20] (-3.5,0)--(0,0) --++ (z1) --++ (z2)  --++ (z3) --++ (z4) --++(3.5,0) --++(0,3)--++(-12.5,0) node[midway,below] {$D^+$}--cycle;

\draw (-3.5,0)--(0,0) node[midway,above] {$l_1$} node {\tiny $\bullet$}
                     --++ (z1) node[midway,above] {$z_1$} node {\tiny $\bullet$}
                     --++ (z2) node[midway,above] {$z_2$} node {\tiny $\bullet$}
                     --++ (z3) node[midway,above] {$z_3$} node {\tiny $\bullet$}
                     --++ (z4) node[midway,above] {$z_4$} node {\tiny $\bullet$}
                     --++(3.5,0) node[midway, above] {$l_2$}  node[above right] {$L_1$} ;
                     
\fill [fill=gray!20] (-3.5,-3)--(0,-3) --++ (z2) --++ (z1)  --++ (z4) --++ (z3) --++(3.5,0) --++(0,-4)--++(-12.5,0) node[midway,above] {$D^-$} --cycle;

 \draw [dashed] (-5,-3)--(-3.5,-3);
\draw [dashed] (9,-2.5)--(10.5,-2.5);
\draw (-3.5,-3)--(0,-3) node[midway,above] {$l'_1$} node {\tiny $\bullet$}
                     --++ (z2) node[midway,above] {$z_2$} node {\tiny $\bullet$}
                     --++ (z1) node[midway,above] {$z_1$} node {\tiny $\bullet$}
                     --++ (z4) node[midway,above] {$z_4$} node {\tiny $\bullet$}
                     --++ (z3) node[midway,above] {$z_3$} node {\tiny $\bullet$}
                     --++(3.5,0) node[midway, above] {$l'_2$}  node[above right] {$L_2$};

\end{tikzpicture}
\caption{Construction of a translation surface with a degree 2 pole.}
\end{figure}
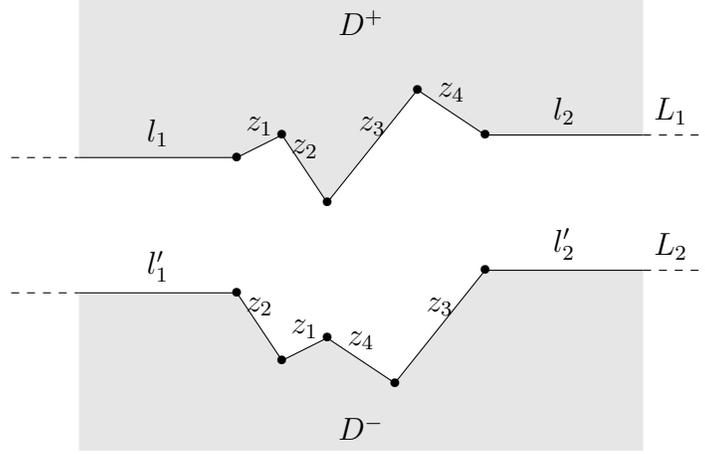
% END TIKZ    END TIKZ    END TIKZ    END TIKZ    END TIKZ    END TIKZ

%\subsubsection{Another case: surface with a single pole of order k> 2.}
%\subsubsection{A surface with several poles of order greater than one}
\subsubsection{General case} \label{inf:zipp:rect}

%convention: nb de disques= d   nb de \zeta_i=dim=n     nb de strip=s,s'

We can generalize the above construction in order to have several poles of any order. Instead of considering two half-planes $D^+,D-$, we will do the same construction starting from $2d$ half-planes,  $s^++s^-$  infinite cylinders, and define identification on their boundary. More precisely:

Let $\zeta\in \mathbb{C}^n$ with positive real part. We consider the following combinatorial data:
\begin{itemize}
\item A collection {$\bf n^+$} of integers $0=n_0^+\leq n_1^+\leq \dots n_d^+ < \dots <n_{d+s^+}^+=n$
\item A collection {$\bf n^-$} of integers $0=n_0^-\leq n_1^-\leq \dots n_d^- < \dots <n_{d+s^-}^-=n$
\item A pair of maps $\pi_t,\pi_b: \mathcal{A}\to \{1,\dots ,n\}$
\item A collection {$\bf d$} of integers $0=d_0< d_1 < d_2 <\dots < d_{r}=d$.
\end{itemize}
The resulting surface will have $r$ poles of order greater than or equal to two, and $s^++s^-$ poles of order 1.

For each $i \in \{0,\dots ,d-1\}$, we consider the basic domains as defined before $D_i^+(\zeta_{\pi_{t}^{-1}(n_i^++1)},\dots ,\zeta_{\pi_{t}^{-1}(n_{i+1}^+)})$ and $D_i^-(\zeta_{\pi_b^{-1}(n_i^-+1)},\dots ,\zeta_{\pi_b^{-1}(n_{i+1}^-}))$.
 For $i\in\{d,\dots ,d+s^+-1\}$, we define $C_i^+(\zeta_{\pi_{t}^{-1}(n_i^++1)},\dots ,\zeta_{\pi_{t}^{-1}(n_{i+1}^+)})$.
  For $i\in\{d,\dots ,d+s^--1\}$, we define $C_i^-(\zeta_{\pi_b^{-1}(n_i^-+1)},\dots ,\zeta_{\pi_b^{-1}(n_{i+1}^-}))$.

%%% TIKZ TIKZ TIKZ TIKZ TIKZ TIKZ TIKZ TIKZ TIKZ TIKZ TIKZ TIKZ
%%%%%%%%% S in H(-3,...)
\begin{figure}[htb]
\begin{tikzpicture}[scale=0.5]
\coordinate (z1) at (1,0.5);
\coordinate (z2) at (1,-1.5);
\coordinate (z3) at (2,2.5) ;
\coordinate (z4) at (1.5,-1) ;
\fill [fill=gray!20] (0,0)--++(3.5,0) --++ (z1) --++ (z2) --++(3.5,0) --++(0,4.5)--++(-9,0)--cycle;

\draw (0,0)--++(3.5,0) node[midway,above] {$l_1$} node {\tiny $\bullet$}
                     --++ (z1) node[midway,above] {$z_1$} node {\tiny $\bullet$}
                     --++ (z2) node[midway,above right] {$z_2$} node {\tiny $\bullet$}
                     --++(3.5,0) node[midway, above] {$l_2$};

\fill [fill=gray!20] (12,-1)--++(3.5,0) --++ (z3) --++ (z4) --++(3.5,0) --++(0,3)--++(-10.5,0)--cycle;

\draw (12,-1)--++(3.5,0) node[midway,above] {$l_3$} node {\tiny $\bullet$}
                     --++ (z3) node[midway,above left] {$z_3$} node {\tiny $\bullet$}
                     --++ (z4) node[midway,above] {$z_4$} node {\tiny $\bullet$}
                     --++(3.5,0) node[midway, above] {$l_4$};
                     
\fill [fill=gray!20] (0,-4)--++(3.5,0) --++ (z2) --++ (z3) --++(2.5,0) --++(0,-5)--++(-9,0)--cycle;

\draw (0,-4)--++(3.5,0)node[midway,below] {$l_1$} node {\tiny $\bullet$}
                     --++ (z2) node[midway,below left] {$z_2$} node {\tiny $\bullet$}
                     --++ (z3) node[midway,below right] {$z_3$} node {\tiny $\bullet$}
                     --++(2.5,0) node[midway, below] {$l_4$};
                     
\fill [fill=gray!20] (12,-3)--++(3.5,0) --++ (z4) --++ (z1) --++(4.5,0) --++(0,-4.5)--++(-10.5,0)--cycle;

\draw (12,-3)--++(3.5,0)node[midway,below] {$l_3$} node {\tiny $\bullet$}
                     --++ (z4) node[midway,below left] {$z_4$} node {\tiny $\bullet$}
                     --++ (z1) node[midway,below right] {$z_1$} node {\tiny $\bullet$}
                     --++(4.5,0) node[midway, below] {$l_2$};
                     
\end{tikzpicture}
\caption{Construction of a translation surface with a degree 3 pole.}
\label{ex:surf:pole:3}
\end{figure}
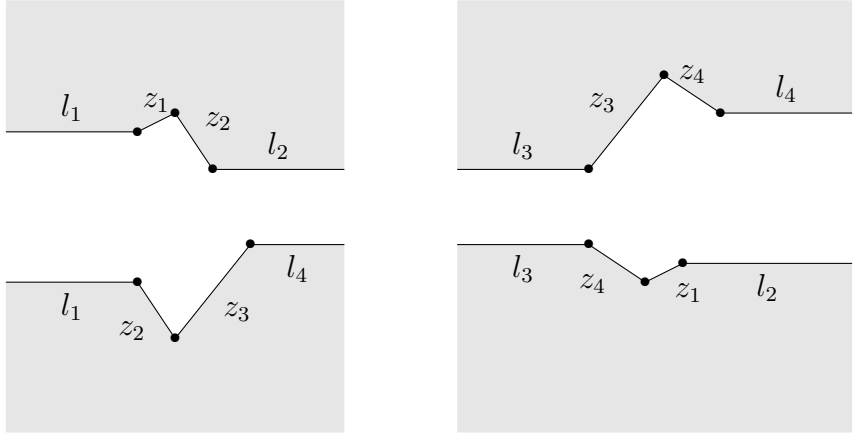
% END TIKZ    END TIKZ    END TIKZ    END TIKZ    END TIKZ    END TIKZ

%%% TIKZ TIKZ TIKZ TIKZ TIKZ TIKZ TIKZ TIKZ TIKZ TIKZ TIKZ TIKZ
%%%%%%%%% S in H(-2,-2,...)
\begin{figure}[htb]
\begin{tikzpicture}[scale=0.5]
\coordinate (z1) at (1,0.5);
\coordinate (z2) at (1,-1.5);
\coordinate (z3) at (2,2.5) ;
\coordinate (z4) at (1.5,-1) ;
\fill [fill=gray!20] (0,0)--++(3.5,0) --++ (z1) --++ (z2) --++(3.5,0) --++(0,4.5)--++(-9,0)--cycle;

\draw (0,0)--++(3.5,0) node[midway,above] {$l_1$} node {\tiny $\bullet$}
                     --++ (z1) node[midway,above] {$z_1$} node {\tiny $\bullet$}
                     --++ (z2) node[midway,above right] {$z_2$} node {\tiny $\bullet$}
                     --++(3.5,0) node[midway, above] {$l_2$};

\fill [fill=gray!20] (12,-1)--++(3.5,0) --++ (z3) --++ (z4) --++(3.5,0) --++(0,3)--++(-10.5,0)--cycle;

\draw (12,-1)--++(3.5,0) node[midway,above] {$l_3$} node {\tiny $\bullet$}
                     --++ (z3) node[midway,above left] {$z_3$} node {\tiny $\bullet$}
                     --++ (z4) node[midway,above] {$z_4$} node {\tiny $\bullet$}
                     --++(3.5,0) node[midway, above] {$l_4$};
                     
\fill [fill=gray!20] (0,-4)--++(3.5,0) --++ (z2) --++ (z3) --++(2.5,0) --++(0,-5)--++(-9,0)--cycle;

\draw (0,-4)--++(3.5,0)node[midway,below] {$l_1$} node {\tiny $\bullet$}
                     --++ (z2) node[midway,below left] {$z_2$} node {\tiny $\bullet$}
                     --++ (z3) node[midway,below right] {$z_3$} node {\tiny $\bullet$}
                     --++(2.5,0) node[midway, below] {$l_2$};
                     
\fill [fill=gray!20] (12,-3)--++(3.5,0) --++ (z4) --++ (z1) --++(4.5,0) --++(0,-4.5)--++(-10.5,0)--cycle;

\draw (12,-3)--++(3.5,0)node[midway,below] {$l_3$} node {\tiny $\bullet$}
                     --++ (z4) node[midway,below left] {$z_4$} node {\tiny $\bullet$}
                     --++ (z1) node[midway,below right] {$z_1$} node {\tiny $\bullet$}
                     --++(4.5,0) node[midway, below] {$l_4$};
                     
\end{tikzpicture}
\caption{Construction of a translation surface with two poles of degree $2$}
\label{ex:surf:2pole:2}
\end{figure}
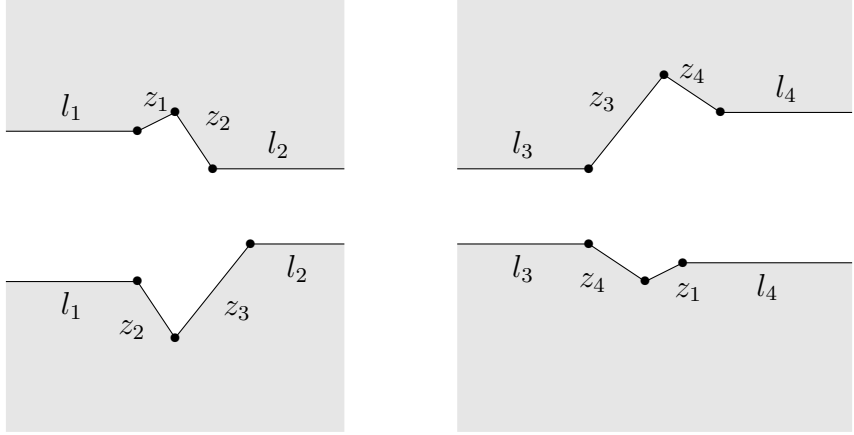
% END TIKZ    END TIKZ    END TIKZ    END TIKZ    END TIKZ    END TIKZ

Then, we glue these domains together on their boundary by translations: 
\begin{itemize}
\item each segment corresponding to a parameter $\zeta_i$ in a $``+''$ domain is glued to the unique corresponding one in a $``-''$ domain.
\item each left line of a domain $D_i^+$ is glued to the left line of the domain $D_i^-$.
\item For each $i\in \{1,\dots ,d\} \backslash \{d_1,\dots , d_r \}$ the right line of the domain $D_i^-$ is glued to one of the domain $D_{i+1}^+$.
\item For each $i=d_k$, $k>0$, the right line of the domain $D_i^-$ is glued to one of the domain $D_{d_{k-1}+1}^+$.
\item For each $C_i^+$ and $C_i^-$, the two vertical lines are glued together.
\end{itemize}

The resulting surface $S$ has no more boundary components, and is a flat surface with poles and conical singularities. It might not be connected for a given combinatorial data. We will consider only those that that give connected surfaces.

Note that such surface is easily decomposes into a finite union of vertical strips and half-planes with vertical boundary (\emph{i.e.} ``infinite rectangle''), that are ``zippered'' on their boundary.

\begin{example}
Figure~\ref{ex:surf:pole:3} shows an example with $d=2$, $s^+=s^-=0$, ${n^+}=(0,2,4)$, ${n^-}=(0,2,4)$, $\pi_t=Id$, $(\pi_b^{-1}(1),\dots ,\pi_b^{-1}(n))=(2,3,4,1)$ and ${\bf d}=(0,2)$. One gets a surface in $\mathcal{H}(-3,5)$. 

Figure~\ref{ex:surf:2pole:2} shows an example with the same data, except that ${\bf d}=(0,1,2)$. One gets a surface in $\mathcal{H}(-2,-2,2,2)$.
\end{example}

\begin{lemma}\label{lem:dim}
Let $S$ be a genus $g$ surface in $\mathcal{H}(n_1,\dots,n_r,-p_1,\ldots,-p_s)$, obtained by the infinite zippered rectangle construction with a continuous parameter $\zeta\in \mathbb{C}^n$. Then 
$$n=2g+r+s-2$$
\end{lemma}

\begin{proof}
By construction, the surface (pole included) is obtained by gluing $s$ disks on their boundary. The resulting surfaces admits a decompositions into cells, with $s$ faces, $n$ edges, and $r$ vertices. So, we have $2-2g=s-n+r$, and the result follows.
\end{proof}

The following proposition will be very useful in the remaing of the paper. It is analogous to the well-known fact that that any translation surface with no vertical saddle connection is obtained by the Veech construction.
\begin{proposition}\label{no:vertical}
Any translation surface with poles with no vertical saddle connection is obtained by the infinite zippered rectangle construction.
\end{proposition}

\begin{proof}
According to the book of Strebel \cite{strebel} Section~11.4, when there are no vertical saddle connections the vertical trajectories are of the following two types: 
\begin{enumerate}
\item lines that converge to a pole in each direction.
\item half-lines starting (or ending) from a conical singularity and converging to a pole on their infinite direction.
\end{enumerate}
Furthermore, the set of non-singular vertical trajectories is a disjoint union of half-planes and of vertical strips $\sqcup_i \mathcal{P}_i \sqcup_j \mathcal{S}_j$. The half-planes have one vertical boundary component, and the strips have two vertical boundary component. 
We choose these half-planes or strips as large as possible, so each boundary component necessarily contains a conical singularity. This singularity is unique for each connected component, otherwise there would be a vertical saddle connection on the surface. 
This number of half-planes and strips is necessarily finite, since there is only a finite number of conical singularities, and each conical singularities is adjacent to a finite number of half-plane or strip.

We cut each half-plane $\mathcal{P}_i$ in the following way: the boundary of $\mathcal{P}_i$ consists of a union of two vertical half-lines starting from the conical singularity. We consider the unique horizontal half-line starting from this singularity and cut $\mathcal{P}_i$ along this half line. It decomposes $\mathcal{P}_i$ into two components $\mathcal{P}_i^+$ (the upper one) and $\mathcal{P}_i^-$ (the lower one). 

We cut each strip $\mathcal{S}_j$ in the following way: the boundary of $\mathcal{S}_i$  has two components, each consists of a union of two vertical half-lines starting from a conical singularity. There is a unique geodesic segment joining these two boundary singularities. We cut $\mathcal{S}_j$ along this segment and obtain two components $\mathcal{S}_j^+$ and $\mathcal{S}_j^-$.

The surface $S$ is obtained as the disjoint union of the $\mathcal{S}_j^\pm$ and $\mathcal{P}_i^\pm$, glued to each other by translation on their boundary components. 
Now we remark that the  $\mathcal{P}_i^+$ and $\mathcal{S}_j^+$ are necessarily glued to each other along their vertical boundary components. Under this identification, $\sqcup  \mathcal{P}_i^+ \sqcup_j \mathcal{S}_j^+$ is a union of subsets of the type $D^+$ and $S^+$, as in the previous construction.

Similarly, gluing together along vertical sides the union of the $\mathcal{S}_j^-$ and $\mathcal{P}_i^-$, one obtains a union of $D^-$ and $C^-$ type subsets.  

So the surface is obtained by the infinite zippered rectangle construction.
\end{proof}

\begin{remark}\label{rq:param}
Note that the parameters $(\zeta_i)_i$ are uniquely defined (once the suitable vertical direction is fixed) and the infinite zipperered rectangle construction defines a triangulation of the surface for which the $(\zeta_i)$ are the local parameters for the strata. Hence, map  $S\to (\zeta_i)_i$ is a local homeomorphism. The corresponding saddle connections form a basis of the relative homology $H_1(S, \Sigma,\mathbb{Z})$, where $\Sigma$ is the set of conical singularities of $S$.
\end{remark}

Note that for any translation surface with poles, the set of saddle connections is at most countable, so it is always possible to rotate the surface in such way that there are no vertical saddle connections. Hence the previous theorem gives a representation for \emph{any} translation surface with poles. One important consequence this theorem is Proposition~\ref{adj:minimal}, which is the analogous version of a key argument in \cite{KoZo}.

\section{Genus one case}\label{genus1} 
\subsection{Connected components}
We first recall the following result in algebraic geometry, which is a consequence of  Abel's theorem.

\begin{NNthm}
Let $\overline{X}=\mathbb{C}/\Gamma$ be a torus and let $D=\sum_i \alpha_i P_i$ be a divisor. Then there exists a meromorphic differential whose divisor is $D$ if and only if $\sum {z_i}\in \Gamma$, where for each $i$,  ${z_i}$ is a representative in $\mathbb{C}$ of $P_i$.
\end{NNthm}

Now we use this theorem to describe the connected components in the genus one case.

\begin{theorem} \label{th:g1:complex}
Let $\mathcal{H}=\mathcal{H}(n_1,\dots  ,n_r,-p_1,\dots ,-p_s) $ be a stratum of genus one meromorphic differentials. Let $d=\gcd(n_1,\dots ,n_r,p_1,\dots  p_s)$, and let $c$ be the number of positive divisors of $d$. Then the number of connected components of $\mathcal{H}$ is:
\begin{itemize}
\item $c$ if $r+s\geq 3$.
\item $c-1$ if $r=s=1$.
\end{itemize}
\end{theorem}

\begin{proof}
%According to the previous theorem, an element in $\mathcal{H}$ is given, up a multiple constant, by a pair $(\overline{X},D)$, where $\overline{X}$ is in the moduli space of genus one Riemann surfaces $\mathcal{M}_1$, and $D$ is a divisor on $\overline{X}$. So, $\mathcal{H}$ (or its projectivization) is a covering of $\mathcal{M}_1$. 

We first assume that $r+s=2$. Then $\mathcal{H}=\mathcal{H}(n,-n)$, for some $n\geq 2$ (the stratum $\mathcal{H}(1,-1)$ is empty). We have $d=n$. 
An element in $\mathcal{H}$ is given, up to a constant multiple, by a pair $(\overline{X},D)$, were $\overline{X}$ is a torus and $D=-nP+nQ$ is a divisor on $\overline{X}$. One can assume that $P=0$, and from the previous theorem, there is only a finite number of possibilities for $Q$, depending only on $n$. Hence, the map $(\overline{X},-n0+nQ)\mapsto (\overline{X},0)$ defines a covering from 
 $P\mathcal{H}$ to $\mathcal{M}_{1,1}$, where $\mathcal{M}_{1,1}$ denotes the moduli space of genus one Riemann surfaces with a marked point.

Fix $\overline{X}_0\in \mathcal{M}_{1,1}$ a regular point, and choose $v_1,v_2$ such that $\overline{X}_0=\mathbb{C}/(v_1 \mathbb{Z}+v_2\mathbb{Z})$. An element $(\overline{X}_0,\omega)\in \mathcal{H}$ is uniquely  defined by the coordinates of $Q$, which are given in a unique way by a complex number of the form $\frac{p}{n}v_1+\frac{q}{n}v_2$, with $(p,q)\neq (0,0)$ and $0\leq p,q<n$. 
Since $\overline{X}_0$ is taken regular, there is a one to one correspondance between  such pairs $(p,q)$ of integers and the elements of $\mathcal{H}$ whose underlying Riemann surface is  $\overline{X}_0$.

The monodromy of the covering $P\mathcal{H}\to \mathcal{M}_{1,1}$ is generated by the two maps  $\phi_1:(p,q)\to (p+q,q) \mod n$ and $\phi_2:(p,q)\to (p,q+p) \mod n$. We remark that $d'=\gcd(p,q,n)$ is invariant by this action and the condition on $(p,q)$ implies that $0<d'<n$. Hence $d'$ is an invariant of the connected components of $\mathcal{H}$. The number of possible $d'$ is $c-1$. We claim that each $(p,q)$ has a (unique) representative modulo these actions of the kind $(d',0)$. To prove the claim, we start from an element $(p,q)$ and do an algorithm similar to Euclid's algorithm. Without loss of generality, one can assume that $p\neq 0$ and $q\neq 0$. Applying $\phi_1^r$ for some suitably chosen $r$, we can obtain $(p',q)$ with $0<p'\leq \gcd(q,n)$. Similarly, we can obtain $(p,q')$ with $0<q<\gcd(p,n)$. So if either $\gcd(q,n)<p$ or $\gcd(p,n)<q$, we obtain $(p',q')$ with $p'+q'<p+q$. Otherwise $p\leq \gcd(q,n)\leq q$ and $q\leq \gcd(p,n)\leq p$. This implies $p=q=d'$, and the result follows.

Now we assume $r+q\geq 3$. We proceed in a similar way as before: we fix $\overline{X}_0=\mathbb{C}/\Gamma$ and a basis $v_1,v_2$ of $\Gamma$. Then a meromorphic differential is given by a a vector $(z_1,\dots z_{r},z_1',\dots ,z_{s}')\in \mathbb{C}^{r+s}$ with pairwise different entries (modulo $\Gamma$), and satisfying the linear equality $\sum_{i=1}^r n_i z_i -\sum_{i=1}^s p_i z_i=p v_1+ q v_2$. One can remark that:
\begin{itemize}
\item For each $(p,q)$ the set of $(z_i)_i,(z_j')_j$ satisfying the previous condition is nonempty and connected. 
\item If we choose other representatives $z_i,z_j'$ for the same differential $\omega$, this changes $(p,q)$ by $(p+\sum_i \alpha_i n_i+\sum_j \beta_j p_j,q+ \sum_{i} \alpha_i' n_i+\sum_j \beta_j' p_j)$, where $(\alpha_i,\beta_j, \alpha_i',\beta_j')$ can be any integers.
\item The action of the two generators of the modular group changes $(p,q)$ by $(p+q,q)$ and $(p,q+p)$ respectively. 
\end{itemize}
 Then, by a proof very similar to the previous one, one can see that $d'=\gcd(p,q,n_1,\dots ,n_r,p_1,\dots ,p_s)$ is an invariant of connected component and one can find representative in each connected component satisfying $(p,q)=(d',0)$. So the number of connected components is precisely $c$. 
 
 Note that the difference with the first case is that any pair $(p,q)\in \mathbb{Z}^2$ is possible.
\end{proof}

\subsection{Flat point of view: rotation number}
The previous section classifies the connected component of the moduli space of meromorphic differentials in the genus one case from a complex analytic point of view. 

But the invariant which is given is not easy to describe in terms of flat geometry.  The next theorem gives an interpretation in terms of flat geometry. 
%, analogous to the ``parity of spin structure'' invariant of Konsevich--Zorich \cite{KoZo}.

Let $\gamma$ be a simple closed curve parametrized by the arc length on a translation surface that avoids the singularities. Then $t\to \gamma'(t)$ defines a map from $\mathbb{S}^1$ to $\mathbb{S}^1$. We denote by $Ind(\gamma)$ the index of this map.

\begin{definition}
Let $S=(X,\omega)\in \mathcal{H}(n_1,\dots ,n_r,-p_1,\dots -p_s)$  be a genus one translation surface with poles.
Let $(a,b)$ be a symplectic basis of the homology the underlying compact Riemann surface $\overline{X}$ and $\gamma_a,\gamma_b$ be arc-length representatives of $a,b$, with no self intersections, and that avoid the zeroes and poles of $\omega$. We define the \emph{rotation number} of $S$ to by:
$$rot(S)=\gcd(Ind(\gamma_a),Ind(\gamma_b),n_1,\dots n_r,p_1,\dots ,p_s)$$
\end{definition}

\begin{theorem}
Let $\mathcal{H}=\mathcal{H}(n_1,\dots ,n_r,-p_1,\dots -p_s)$ be a stratum of genus 1 meromorphic differentials. The rotation number is an invariant of connected components of $\mathcal{H}$. 

Any positive integer $d$ which divides $\gcd(n_1,\dots ,n_r,p_1,\dots p_s)$ is realised by a unique connected component of $\mathcal{H}$, except for the case $\mathcal{H}=\mathcal{H}(n,-n)$ where $d=n$ doesn't occur.
\end{theorem}

\begin{proof}
Let $(a,b)$ be a symplectic basis of $H_1(\overline{X},\mathbb{Z})$. 
Let $\gamma_a, \gamma_a'$ be representatives of $a$ that are simple closed curves and don't contain a singularity. Since $\overline{X}$ is a torus, $\gamma_a$ and $\gamma_a'$ are homotopic as curves defined on $\overline{X}$. The index of $\gamma_a$ doesn't change while we deform $\gamma_a$ without crossing a pole or a zero.
It is easy to see that when crossing a singularity of order $k\in \mathbb{Z}$, the index is changed by adding $\pm k$. Hence the rotation number only depend on the homology class of $a$ and $b$. 

If $\gamma_a$ and $\gamma_b$ intersects in one point, then there is a standard way to construct a simple closed curve representing $a\pm b$. Its index is $Ind(\gamma_a)\pm Ind(\gamma_b)$, and we obtain representatives of the symplectic basis $(a\pm b,b)$ (or $(a,a\pm b)$). The rotation number doesn't change by  this operation.
With this procedure, we can obtain any other symplectic basis of $\overline{X}$.

Hence the rotation number is well defined for a given element of $\mathcal{H}$. Also,   it is invariant by deforming $(\overline{X},\omega)$ inside the ambiant stratum, since by continuous deformation, we can keep track of a pair of representatives of a basis, and the indices are constant under continuous deformations. 

To prove the last part of the theorem, we remark that a surface in $\mathcal{H}(n,-p_1,\dots ,-p_s)$ obtained from $\mathcal{H}(n-2,-p_1,\dots ,-p_s)$ by bubbling a handle with parameter $k\in \{1,\dots ,n-1\}$ has a rotation number equal to  $\gcd(k,p_1,\dots ,p_s)$ by a direct computation.  Since $k<n$, we have $\gcd(k,p_1,\dots ,p_s)<n$ so $n$ is never a rotation number. Now we break up the singularity of order $n$ to get $r$ singularities of order $n_1,\dots ,n_r$. Since this doesn't change the metric outside a small neighborhood of the singularity of order $n$, we obtain a rotation number equal to $\gcd(k,n_1,\dots ,n_r,p_1,\dots ,p_s)$.

The previous construction gives at least as many connected component as the number given by  Theorem~\ref{th:g1:complex}. So, we see each rotation number is realized by a unique component, and that this component is realized by the bubbling a handle construction.
\end{proof}

Note that the last two paragraphs of the proof of  the last theorem gives the following description of the connected components of the minimal strata in genus one.

\begin{proposition} \label{g1:cyl}
Let  $\mathcal{H}=\mathcal{H}(n,-p_1,\dots ,-p_s)$  be a minimal stratum of genus one meromorphic differentials. 
Any connected component of $\mathcal{H}$ is of the form $\mathcal{H}_0 \oplus k$, for some $1\leq k \leq n-1 $, where $\mathcal{H}_0$ is the connected stratum $\mathcal{H}(n-2,-p_1,\dots ,-p_s)$.

Also, for $1\leq k_1,k_2 \leq n-1$ we have: $$\mathcal{H}_0\oplus k_1=\mathcal{H}_0\oplus k_2$$  if and only if $\gcd(k_1,p_1,\dots ,p_s)=\gcd(k_2,p_1,\dots ,p_s)$.
\end{proposition}

\begin{remark}
It is shown in the appendix that there are some translation surface with pole that do not contain any closed geodesic.
\end{remark}

\section{Spin structure and hyperelliptic components}\label{invariants}
Recall that in the classification of the connected components of strata of the moduli space of Abelian differentials \cite{KoZo}, the connected components are distinghished by two invariants. 
\begin{itemize}
\item ``Hyperelliptic components'': there are some connected components whose corresponding translation surfaces all have an extra symmetry.
\item ``The parity of the spin structure'', which is a complex invariant that can be expressed in terms of the flat geometry by a simple formula.
\end{itemize}
\subsection{Hyperelliptic components}
\begin{definition}
A translation surface with poles $S$ is said to be \emph{hyperelliptic} if there exists an isometric involution $\tau:S\to S$ such that $S/\tau$ is a sphere. Equivalently, the underlying Riemann surface $\overline{X}$ is hyperelliptic and the hyperelliptic involution $\tau$ satisfies $\tau^* \omega=-\omega$.
\end{definition}
\begin{remark}
In the case of Abelian differentials, if the underlying Riemann surface is hyperelliptic, then the translation surface is hyperelliptic since there are no nonzero holomorphic one forms on the sphere. In our case, similarly to the case of quadratic differentials, the underlying Riemann surface might be hyperelliptic, while the corresponding translation surface is not.
\end{remark}

\begin{proposition}\label{list:hyp}
Let $n,p$ be positive integers with $n\geq p$. The following strata admit a connected component that consists only of hyperelliptic translation surfaces.
\begin{itemize}
\item $\mathcal{H}(2n,-2p)$
\item $\mathcal{H}(2n,-p,-p)$
\item $\mathcal{H}(n,n,-2p)$
\item $\mathcal{H}(n,n,-p,-p)$
\end{itemize}
Furthermore, any stratum that contains an open set of flat surfaces with a nontrivial isometric involution is in the previous list for some $n\geq p\geq 1$.

\end{proposition}

\begin{proof}
Let $\mathcal{H}$ be a stratum  and $\mathring{\mathcal{H}}^{hyp}\subset \mathcal{H}$ the interior of the set of elements of $\mathcal{H}$ that admit a nontrivial isometric involution. 

Given a combinatoria datum $\sigma=(\mathbf{n}^+,\mathbf{n}^-,\pi_t,\pi_b,\mathbf{d})$ that defines an infinite zippered rectangle construction, we denote by $\mathcal{C}_\sigma$ the set of flat surfaces that are obtained by this construction with parameter $\sigma$, up to a rotation. Clearly, $\mathcal{C}_\sigma$ is open and connected. 

We claim that for each $\sigma$, the intersection between $\mathcal{C}_\sigma$ and $\mathring{\mathcal{H}}^{hyp}$ is either $\mathcal{C}_\sigma$ or empty. Indeed, choose a generic parameter $\zeta$ for the infinite zippered rectangle construction, such that the corresponding surface $S(\sigma,\zeta)$ is in $\mathring{\mathcal{H}}^{hyp}$. Let $D^+(z_1,\dots ,z_k)\subset S(\sigma,\zeta)$ be a half-plane of the construction. Then, $\zeta$ being generic, an isometric involution $\tau$ will necessarily send the segment corresponding to $z_i$ to itself. Hence if $\tau$ is not the identity, it is easy to see that the set $D^+(z_1,\dots ,z_k)$ will be sent to $D^-(z_k,z_{k-1},\dots , z_1)$, and therefore, we can define a similar involution for \emph{any} value of $z_1,\dots ,z_k$. Since this argument is valid for any $D^{\pm}$ and $C^{\pm}$ components, we see that all flat surfaces obtained by the infinite zippered rectangle construction with combinatorial datum $\sigma$ have a nontrivial isometric involution. This proves the claim.

Now we remark that, by Proposition~\ref{no:vertical}, $\mathcal{H}=\cup_\sigma \mathcal{C}_\sigma$, where the union is taken on all $\sigma$ that corresponds to $\mathcal{H}$. The previous claim implies that $\mathring{\mathcal{H}}^{hyp}$ and its complement in $\mathcal{H}$ are both unions of some $\mathcal{C}_\sigma$, so if $\mathring{\mathcal{H}}^{hyp}$ is nonempty, it is a connected component of $\mathcal{H}$.

Now we check that if $\mathring{\mathcal{H}}^{hyp}$ is not empty, then the stratum $\mathcal{H}$ is in the given list, \emph{i.e.} there is either one even degree zero (resp. pole) or two zeroes (res. poles) of equal degree.
Let $\zeta_1,\dots ,\zeta_n$ be the continuous data in the infinite zippered rectangle construction for an element $S$ in $\mathring{\mathcal{H}}^{hyp}$. The above condition implies that for each $\zeta_i$, the middle of the corresponding segment in the surface is a fixed point for the involution $\tau$. So, there are at least $n$ fixed points. Let  $r$ be the number of conical singularities, let $s$ be the number of poles and let $g'$ be the genus of $S/\tau$. We must have $\#(Fix(\tau))=2g+2-4g'$, and $2g+r+s-2=n\leq \#(Fix(\tau))$ (see Lemma~\ref{lem:dim}). Since $r,s>1$, this implies $g'=0$, so $S$ is hyperelliptic, and $\#(Fix(\tau))-n=4-r-s$. The fixed points of $\tau$ in $\overline{X}$ that do not correspond to the middle of a $z_i$ segment are necessarily either conical singularities or pole. 

The above combinatorial condition on the infinite zippered rectangle construction implies that $S$ has either two equal degree poles that are interchanged by $\tau$ or one pole of even degree that is preserved by $\tau$. So the condition $\#(Fix(\tau))-n=4-r-s$ implies that either there is one conical singularity which is fixed by $\tau$, or there are two singularities $P_1, P_2$ that are not fixed by $\tau$. By a similar argument as in the proof of Proposition~\ref{adj:minimal}, $P_1,P_2$ are the endpoints of a saddle connection corresponding to a parameter $\zeta_i$, so they are interchanged by $\tau$, hence they are of the same degree. Therefore, the stratum is necessarily one of the given list.

The last step of the proof is to check that for the strata given in the statement, $\mathring{\mathcal{H}}^{hyp}$ is nonempty. This is an elementary check by using the infinite zippered rectangle construction that satisfies the previous condition. 

\end{proof}

\subsection{The parity of the spin structure for translation surfaces with even singularities}\label{subsec:spin}
\subsubsection{Spin structures on a surface}
There are two equivalent definitions of spin structure for a compact Riemann surface $X$ commonly used.

The first one is topological: let  $P$ be the $\mathbb{S}^1$ bundle of directions of nonzero tangent vectors to $X$. A \emph{spin structure} on $X$ is a two-to-one covering $Q\to P$, whose restriction to a $\mathbb{S}^1$ fiber is the usual double covering $\mathbb{S}^1\to \mathbb{S}^1$. It is equivalent to a morphism $\xi:H_1(P,\mathbb{Z}/2\mathbb{Z})\to \mathbb{Z}/2\mathbb{Z}$ such that the image of the cycle $z$ corresponding to a fiber is one. Indeed, in this case the monodromy $\pi_1(P)\to \mathbb{Z}/2\mathbb{Z}$  factors  to a map $\xi: H_1(P, \mathbb{Z}/2\mathbb{Z})\to \mathbb{Z}/2\mathbb{Z}$.

The second equivalent definition comes from algebraic geometry (see \cite{Ati}). A \emph{theta characteristic}, is a solution of the equation $2D=K$ in the divisor class group, where $K$ is the canonical divisor. Equivalently, it is  a complex line bundle $L$ such that $L\otimes L\sim \Omega_1(X)$. For such $L$, Atiyah and Munford showed independently \cite{Ati,Mum} that the dimension modulo 2 of the vector space of holomorphic sections of $L$ is invariant by deformation of the complex structure. This is the \emph{parity of the spin structure}.

In  \cite{Johnson}, Johnson provides a topological way to compute this invariant. He first constructs a lift $C\mapsto \widetilde{C}$ from $H_1(X,\ZZ)$ to $H_1(P,\ZZ)$. We refer \cite{Johnson} for details on the construction of the lift. In our case, it is enough to observe that when $C=[\gamma]$ is the class of a simple closed curve, the lift of $C$ is $\widetilde{C}=[\ha{\gamma}]+z$  where $\ha{\gamma}$ is the natural lift obtained by framing $\gamma$ with its speed vector $\gamma'$, and $z$ is the class of a $\mathbb{S}^1$ fiber.

The composition of this lift with the map $\xi$ gives a quadratic form $\Omega:H_1(X,\ZZ) \to \ZZ$, $\Omega(C):=\xi(\widetilde{C})$. Johnson then shows that the parity of the spin structure is equal to the Arf invariant of $\Omega$, \emph{i.e.} for a symplectic basis $(a_1,b_1),\dots ,(a_g,b_g)$ of $H_1(X,\ZZ)$, the parity of the spin structure is
$$\sum_{i=1}^g \Omega(a_i)\Omega(b_i).$$

\subsubsection{The parity of the spin structure for translation surfaces with even singularities}

A translation surface $(X,\omega)$ with even poles and zeroes naturally gives a spin structure on $X$ in the following way: let $(\omega)=\sum_i 2n_i N_i-\sum_j 2p_jP_j $ the divisor associated to $\omega$. Then $D=\sum n_i N_i-\sum_j P_j$ satisfies $2D=K$.  From the results of Atiyah and Munford in \cite{Ati, Mum}, it follows that the parity of the spin structure is a locally constant function on the strata where it is defined. Hence, it is an invariant of connected components, for strata with even poles and zeroes.

The parity of the spin structure can be easily computed by using Johnson's construction.
Following \cite{KoZo}, it is easy to see that the corresponding map $\Omega$ satisfies, for $\gamma$ a simple closed curve, $\Omega([\gamma])=ind(\gamma)+1$. Hence, the parity of the spin structure for a translation surface (with poles) is:
$$\sum_{i=1}^g (ind(a_i)+1)(ind(b_i)+1).$$

\subsection{The parity of the spin structure for translation surfaces with only two simple poles}\label{spin:poles:simples}

Let $S=(X,\omega)\in \mathcal{H}(2n_1,\dots ,2n_r,-1,-1)$ be a translation surface with zeroes of even order and a pair of simple poles (and no other poles). Since there are odd degree singularities, $\omega$ does not define a spin structure on $X$. However, one can still define a topological invariant, that will be the parity of the spin structure on another surface $S'$.

Recall that a neigborhood a simple pole is an infinite cylinder. Choose a pair of waist curves $\gamma_1,\gamma_2$ on each cylinder associated to the two simple poles. Since there are no other poles than the pair of simple poles, the two have opposite residues by Stokes' theorem, 
hence  $\gamma_1,\gamma_2$ are isometric. Now we cut the surface $S$ along $\gamma_1$ and $\gamma_2$. We obtain a compact translation surface with two geodesic boundary components. The condition on the residues implies that gluing together these boundary components by a translation gives a translation surface $S'$, where the pair of infinite cylinders in $S$ corresponds to a finite cylinder $C\subset S'$. The surface $S'$ belongs to the stratum $\mathcal{H}(2n_1,\dots ,2n_r)$ where the spin structure was defined by Kontsevich and Zorich. Note that other choices for $\gamma_1,\gamma_2$, and for the gluing operation only change the length and twist of $C$, hence gives the same connected component of $\mathcal{H}(2n_1,\dots ,2n_r)$.
We will call \emph{the parity of the spin structure of S} the parity of the spin structure of the corresponding translation surface $S'$.

\begin{remark}
Note that one can also define the parity of the spin structure of $S=(X,\omega)$ in a more algebro-geometric way: we consider the stable curve $\overline{X}$ obtained by gluing together the two poles.
 In \cite{Cornalba}, Cornalba extends to a large class of stable curves (including this case) the notion of spin structure, and shows that the parity of the spin structure is invariant by deformations. 
The one form $\omega$ on $X$ can then be used to define on $\overline{X}$ a spin structure.
\end{remark}

%So, in this way we get an invariant for the strata of the kind $\mathcal{H}(n_1,\dots ,n_r,-1,-1)$ with all $n_i$ even.

\section{Higher genus case: minimal stratum} \label{sec:minimal}

Recall that a minimal stratum correspond to the case where there is only one conical singularity (and possibly several poles). As in the papers of Kontsevich-Zorich  \cite{KoZo} and Lanneau  \cite{La:cc}, we first describe the connected components of minimal strata. The idea is similar: show that each such strata is obtained by bubbling $g$ cylinders and compute the connected components in this case.

The first step is to find a surface obtained by bubbling a handle. In \cite{KoZo} and in \cite{La:cc} is used a rather combinatorial argument. A similar approach is possible in our case by using the infinite zippered rectangle construction, but this is quite technical. 
Another possibility is to reduce the problem to the genus one case for which it was proven in Section~\ref{genus1} that any minimal stratum contains a surface obtained by bubbling a handle.

\begin{proposition} \label{exists:cyl}
Let $\mathcal{C}$ be a connected component of the stratum $\mathcal{H}(n,-p1,\ldots,-p_s)$. We assume that the genus $g$ is nonzero. Then, there exists a flat surface in $\mathcal{C}$ which is obtained by bubbling a handle from a genus $g-1$ flat surface.
\end{proposition}

\begin{proof}
We start from a surface in $\mathcal{C}$ obtained by the infinite zippered rectangle construction. It is defined by a combinatorial data and a continous parameter $\zeta\in\mathbb{C}^n$, with $n=2g+s-1$.

Each $\zeta_i$ defines a closed geodesic path $\gamma_i$ joining the conical singularity to itself. The intersection number between any two such paths is $0$ or $\pm 1$.  The genus is higher than zero and $\{\gamma_1,\ldots,\gamma_n\}$ generates the whole homology space $H_1(S,\mathbb{Z})$ since the complement is a union of punctured disks. Hence, there is a pair $\gamma_i,\gamma_j$ whose intersection number is one. 

Now we shrink $\zeta_i,\zeta_j$ until they are very small compared to all the other parameters. Then, we observe that a neighborhood of $\gamma_i,\gamma_j$ is isometric to the complement of a neighborhood of a pole for a surface in $\mathcal{H}(n,-n)$. Then, deforming suitably the surface,  using Proposition~\ref{g1:cyl}, one obtains the desired result.
\end{proof}

We recall the notation introduced in Section~\ref{bubbling}. Let $\mathcal{C}$ is a connected component of a minimal stratum $\mathcal{H}(n,-p_1,\dots ,-p_s)$. Let $s\in \{1,\dots ,n+1\}$. The set $\mathcal{C}\oplus s$ is the connected component of the stratum $\mathcal{H}(n+2,-p_1,\dots ,-p_s)$ obtained by bubbling a handle after breaking the singularity of order $n$ into two singularities of order $(s-1)$ and $(n+1-s)$.

The proposition that follows uses roughly the same arguments as in \cite{KoZo} and \cite{La:cc}. The only difference is the case when $n$ is odd, which does not occur for Abelian or quadratic differentials. 

\begin{proposition} \label{min:strat:upper:bound}
Let $\mathcal{H}(n,-p_1,\dots ,-p_s)$ be a stratum of meromorphic differentials genus $g\geq 2$ surfaces, and denote by $\mathcal{C}_0$ the unique component of $\mathcal{H}(n-2g,-p_1,\dots ,-p_s)$. The following holds:
\begin{itemize}
\item If $n$ is odd, the stratum $\mathcal{H}(n,-p_1,\dots ,-p_s)$ is connected.
\item If $n$ is even, the stratum $\mathcal{H}(n,-p_1,\dots ,-p_s)$ has at most three connected components which are in the following list:
\begin{itemize}
\item $\mathcal{C}_0\oplus  (\frac{n-2g}{2}+1) \oplus  (\frac{n-2g}{2}+2)\oplus\dots \oplus (\frac{n-2g}{2}+g)$
\item $\mathcal{C}_0\oplus 1\oplus\dots \oplus 1\oplus 1$
\item $\mathcal{C}_0\oplus 1\oplus\dots \oplus 1\oplus 2$
\end{itemize}

\end{itemize}
\end{proposition}
\begin{proof}
Let $\mathcal{C}$ be a connected component of $\mathcal{H}(n,-p_1,\dots ,-p_s)$.
By proposition \ref{exists:cyl}, there exist integers $s_1,\dots ,s_g$, such that:
$$\mathcal{C}=\mathcal{C}_0\oplus s_1\oplus\dots \oplus s_g$$
and for each $i\in \{1,\dots ,g\}$, $1\leq s_i\leq n-2g-2+2i+1$, since at Step~$i$, the handle corresponding to $s_i$ is bubbled on a zero of degree $n-2g+2(i-1)$.

We assume for simplicity that $g=2$, and $(s_1,s_2)\neq ( \frac{n-2g}{2}+1, \frac{n-2g}{2}+2)$. 
Using operations $(1)$ and $(3)$ of Lemma~\ref{lemme:lanneau}, one can assume that $1\leq s_1\leq s_2\leq s_1+1$. Then, if $1\neq s_1$, using operations $(1)$, $(2)$, $(3)$ and $(1)$ (in this order), we have $\mathcal{C}_0\oplus s_1\oplus s_2=\mathcal{C}_0\oplus (s_1-1)\oplus (s_2-1)$. Repeating the same sequence of operations, we see that $\mathcal{C}$ is one of the following:
\begin{itemize}
\item $\mathcal{C}_0\oplus  (\frac{n-2g}{2}+1) \oplus  (\frac{n-2g}{2}+2)$
\item $\mathcal{C}_0\oplus 1\oplus 1$
\item $\mathcal{C}_0\oplus 1\oplus 2$
\end{itemize}

If $n$ is odd, then the first case doesn't appear. By operation $(4)$ of Lemma~\ref{lemme:lanneau},  we have $$\mathcal{C}_0\oplus s_1\oplus s_2=\mathcal{C}_0\oplus s_1\oplus ((n-2g+2)+2-s_2)$$
so we can assume that $s_1$ and $s_2$ are of the same parity. Then, using the previous argument, we have:
$$\mathcal{C}=\mathcal{C}_0\oplus 1\oplus 1$$
The case $g>2$ easily follows. 
\end{proof}

The above proposition uses purely local constructions in a neighborhood of a singularity. 
The next proposition explains why the existence of suitable poles (at infinity) will ``kill'' some components.

\begin{proposition} \label{min:strat:odd:poles:or:non:hyp}
Let $\mathcal{H}(n,-p_1,\dots ,-p_s)$ be a stratum of meromorphic differentials on surfaces of genus $g\geq 2$ with $n$ even and $s\geq 2$, and denote by $\mathcal{C}_0$ the unique component of $\mathcal{H}(n-2g,-p_1,\dots ,-p_s)$. The following holds:
\begin{enumerate}
\item If there is a odd degree pole and $\sum_i p_i>2$, then:
$$\mathcal{C}_0\oplus 1\oplus\dots \oplus 1=\mathcal{C}_0\oplus 1\oplus\dots \oplus 1 \oplus 2 $$
\item If $s> 2$ or $p_1\neq p_2$, then:
$$\mathcal{C}_0\oplus  \left(\frac{n-2g}{2}+1\right) \oplus  \dots \oplus \left(\frac{n-2g}{2}+g\right) = \mathcal{C}_0 \oplus 1\oplus\dots \oplus 1\oplus s$$
for some $s\in \{1,2\}$.
\end{enumerate}
\end{proposition}

\begin{proof}
%\begin{enumerate}

Case (1).\\
Note that $s\geq 2$ implies that we necessarily have $\sum_{i} p_i\geq 2$.

From Proposition~\ref{g1:cyl}, $\mathcal{C}_0\oplus 2=\mathcal{C}
_0\oplus k$ if and only if $\gcd(k,p_1,\dots ,p_s)=\gcd(2,p_1,\dots ,p_s)$. So, if there is an odd degree pole, $\gcd(2,p_1,\dots ,p_s)=1=\gcd(1,p_1,\dots ,p_s)$, hence 
$$(\mathcal{C}_0\oplus 1)\oplus 1\dots \oplus 1=(\mathcal{C}_0\oplus 2)\oplus 1\dots \oplus 1= \mathcal{C}_0\oplus 1\dots \oplus 1 \oplus 2,$$
which concludes the proof. Note that $\mathcal{C}_0\oplus 2$ is well defined because $\sum_{i}p_i>2$.
\medskip

\noindent
Case (2).\\
As before, we use the classification in genus one.
 Since $n-2g-\sum_{i} p_i=-2$, we have $\frac{n-2g}{2}+1=\frac{\sum_i p_i}{2}$. If $s> 2$ or $p_1\neq p_2$, then there exists $i\in \{1,\dots ,s\}$ such that $\frac{n-2g}{2}+1> p_i$, so $\gcd(\frac{n-2g}{2}+1,p_1,\dots ,p_s)< \frac{n-2g}{2}+1$, hence there exists $k<\frac{n-2g}{2}+1$ such that $\mathcal{C}_0\oplus (\frac{n-2g}{2}+1)=\mathcal{C}_0\oplus k$. So we have 
$$\mathcal{C}_0\oplus (\frac{n-2g}{2}+1)\oplus (\frac{n-2g}{2}+2)\oplus \dots 
= \mathcal{C}_0 \oplus k \oplus (\frac{n-2g}{2}+2) \oplus \dots  $$
Then, as in the proof of  Proposition~\ref{min:strat:upper:bound}, 
$$\mathcal{C}_0 \oplus k \oplus (\frac{n-2g}{2}+2) \dots  \oplus (\frac{n-2g}{2}+g) =
\mathcal{C}_0 \oplus 1\oplus \dots \oplus 1\oplus s $$
for some $s\in \{1,2\}$.
%\end{enumerate}
\end{proof}

Putting together the last two propositions and the invariants, we have the following theorem.

\begin{theorem}\label{th:str:min}
Let $\mathcal{H}=\mathcal{H}(n,-p_1,\dots ,-p_s)$ be a minimal stratum of meromorphic differentials on genus $g\geq 2$ surfaces. We have:
\begin{enumerate}
\item If $n$ is even and $s=1$, then $\mathcal{H}$ has  two connected components if $g=2$ and $p_s=2$, three otherwise.
\item If $\mathcal{H}=\mathcal{H}(n,-p,-p)$, with $p$ even, then $\mathcal{H}$ has three connected components. 
\item If $\mathcal{H}=\mathcal{H}(n,-1,-1)$, then $\mathcal{H}$ has three connected components for $g>2$, two otherwise.
\item If $\mathcal{H}=\mathcal{H}(n,-p,-p)$, with $p\neq 1$ odd, then $\mathcal{H}$ has two connected components. 
\item If all poles are of even degree and we are not in the previous case, then $\mathcal{H}$ has two connected components.
\item In the remaining cases, $\mathcal{H}$ is connected.
\end{enumerate}

\end{theorem}

\begin{proof}
From Proposition~\ref{min:strat:upper:bound}, when $n$ is odd, which is part of Case~(6), $\mathcal{H}$ is connected. So we can assume that $n$ is even. 
Let $\mathcal{C}$ be a connected component of $\mathcal{H}$. Let $\mathcal{C}_0$ be the (connected) genus 0 stratum $\mathcal{H}(n-2g,-p_1,\dots ,-p_s)$. 
From  Proposition~\ref{min:strat:upper:bound}, we have one of the three following possibilities.
\begin{enumerate}
\item[a)] $\mathcal{C}=\mathcal{C}_0\oplus (\frac{n-2g}{2}+1) \oplus (\frac{n-2g}{2}+2)\oplus\dots \oplus \frac{n}{2} $
\item[b)] $\mathcal{C}=\mathcal{C}_0\oplus 1 \oplus \dots \oplus 1 \oplus 1 $
\item[c)] $\mathcal{C}=\mathcal{C}_0\oplus 1 \oplus \dots \oplus 1\oplus 2$
\end{enumerate}

When $\mathcal{H}=\mathcal{H}(n,-p)$ or $\mathcal{H}=\mathcal{H}(n,-p,-p)$, it is easy to see that case $a)$ corresponds to a hyperelliptic connected component, while case $b)$ does not, and neither $c)$ (except for the case $n-2g=0$ and $g=2$, where $a)$ and $c)$ are the same). 

When all degree of zeroes (and poles) are even, then Lemma~11 in \cite{KoZo} shows that cases b) and c) correspond to different spin structures, so are a different connected components. This is also true for $\mathcal{H}(n,-1,-1)$ by Section~\ref{spin:poles:simples}.

The arguments of the two previous paragraphs proves the result for Cases (1), (2) and (3). Remark that $(n-2g)-\sum_i p_i=-2$.

For Case $(4)$, Proposition~\ref{min:strat:odd:poles:or:non:hyp} shows that there are at most two connected components. Since $n-2g=2p-2>0$, Case~$a)$ corresponds to a hyperelliptic component while $b)$and $c)$ do not correspond to a hyperelliptic component. So there are at least two components. Since there are odd degree poles, $b)$ and $c)$ correspond to the same component by Proposition~\ref{min:strat:odd:poles:or:non:hyp}. So there are two components.

For Case $(5)$, Proposition~\ref{min:strat:odd:poles:or:non:hyp} shows that $a)$ is in the same connected component as $b)$ or $c)$, while Lemma~11 in \cite{KoZo} shows that $b)$ and $c)$ have different spin structures.

For Case $(6)$, with $n$ is even: this corresponds to having at least one odd pole, and either at least three poles or two poles of different degree. Then a direct application of Proposition~\ref{min:strat:odd:poles:or:non:hyp} shows that $a)$, $b)$ and $c)$ are the same connected component.

This concludes the proof.

%We claim that if there is at least one odd pole and we are not in the ``two simple pole case'', $b)$ and $c)$ corresponds to the same connected component. For simplicity, we first assume that there is a odd pole of degree greater than or equal to $3$. 
%To do this, we can start from any surface $S_0$ in $\mathcal{C}_0=\mathcal{H}(n',-p_1\dots ,-p_s)$, with $p_1\geq 3$ odd and do the bubbling construction. We define the surface $S_0$ in the following way: start from two surfaces $S_{01}\in \mathcal{H}(n'-p_1,-p_2,\dots ,-p_s)$ and $S_{02}$ in $\mathcal{H}(p_1-2,-p_1)$ (these are two nonempty genus zero strata).  Choose an infinite horizontal line $l_1$ that starts from the conical singularity on $S_{01}$ in the right direction, and a line $l_2$ that start from the conical singularity on $S_{02}$ in the left direction. Then, cut on $S_{01}$ an infinite rectangle of side $l_1$ and of vertical side a small segment of length $\varepsilon$, and glue together by a translation the two infinite horizontal boundary segment of the resulting surface. We do the same for $S_{02}$ and then glue together the two small vertical segments. We get a surface in $\mathcal{H}(n',-p_1,\dots ,-p_s)$.

\end{proof}

\section{Higher genus case: nonminimal strata} \label{sec:nonminimal}
The remaining part of the paper uses similar arguments as in Sections~5.2--5.4 in \cite{KoZo}. We quickly recall the three main steps.
\begin{itemize}
\item Each stratum is adjacent to a minimal stratum, and we can bound the number of connected components of a stratum by the number of connected components of the corresponding minimal one.
\item We construct paths in suitable strata with two conical singularities that join the different connected components of the minimal stratum.
\item We deduce from the previous arguments upper bounds on the number of connected component of a stratum, lower bounds are given by the topological invariants.
\end{itemize}

The following proposition is analogous to Corollary~4 in \cite{KoZo}. It is proven there by constructing surfaces with a one cylinder decomposition. Such surfaces never exist in our case, we use the infinite zippered rectangle construction instead.

\begin{proposition}\label{adj:minimal}
Any connected component of a stratum of meromorphic differentials is adjacent to the minimal stratum obtained by collapsing together all the zeroes.
\end{proposition}

\begin{proof}
Let $S$ be in a stratum $\mathcal{H}$ of meromorphic differentials. We prove the result by induction on the number of conical singularities of $S$. 
We can assume that $S$ is obained by the previous construction. By connectivity of $S$, there is a $D^\pm$ component or a $C^\pm$ component that contains two different conical singularities  on its boundary, hence, there is a parameter $\zeta_i$ whose corresponding segment on that component joins two different conical singularities.  %$P_1$ and $P_2$. 
The segment is on the boundary of two components. Assume for instance, that it is a $D^+$ and a $C^-$ component. Now we just need to check that the surface obtained by shrinking $\zeta_i$ to zero is nondegenerate. Hence it will correspond to an element in a stratum with one  conical singularity less. 
The set $D'^+$ obtained by shrinking $\zeta_i$ to zero from $D^+$ is still a domain as defined in Section~\ref{basic:domains}. The set $C'^-$ obtained by shrinking $\zeta_i$ to zero from $C^-$ is also a domain as defined in Section~\ref{basic:domains} except if we have $C^-=C^-(\zeta_i)$. But in this case, since the two vertical lines of $C^-$ are identified together, the two endpoints of the segment defined by $\zeta_i$ are necessarily the same singularity, contradicting the hypothesis.

So, in any case, we obtain a surface $S'$ with fewer conical singularities. 
\end{proof}

The following proposition is analogous to Corollary~2 in  \cite{KoZo}, and is the first  step of the proof described in the beginning of this section. The proof of Kontsevich and Zorich uses a deformation theory argument. We propose a proof that uses only flat geometry.

\begin{proposition} \label{prop:upper:bound}
The number of connected component of a stratum is smaller than or equal to the number of connected component of the corresponding minimal stratum.
\end{proposition}

\begin{proof}
From the previous proposition, any connected component of a stratum $\mathcal{H}=\mathcal{H}(k_1,\dots ,k_r,-p_1,\dots, -p_s)$ is adjacent to a minimal stratum $\mathcal{H}^{min}= \mathcal{H}(k_1+,\dots +k_r,-p_1,\dots ,-p_s)$ by collapsing zeroes. It is enough to show that if $(S_n),\ (S'_n)$ are two sequences in $\mathcal{H}$ that converge to a surface $S\in \mathcal{H}^{min}$, then $S_n$ and $S_n'$ are in the same connected component of $\mathcal{H}$ for $n$ large enough.

%There exists $\eta>0$ such that any curve $\gamma$ on $S$ which is non homothopic to zero has length greater than $\eta$.

By definition of the topology on the moduli space of meromorphic differentials , for $n$ large enough, 
%this is also true for curves in $S_n,S_n'$. The conical singularities of $S_n$ (resp. $S_n'$) are pairwise at a distance at most $\varepsilon<< \eta$. 
the conical singularities of $S_n$ (resp. $S_n'$) are all in a small disk $D_n$ (resp. $D_n'$)  which is embedded in the surface $S_n$ (resp. $S_n'$), and whose boundary is a covering of a metric circle.

Note that $D_n$ and $D_n'$ can be chosen arbitrarily small if $n$ is large enough, and we can assume that they have isometric boundaries.  Replacing $D_n$ by a disk with a single singularity, one obtains a translation surface $\widetilde{S}_n$ which is very near to $S$, hence in the same connected component, and similarly for $S_n'$.

Now we want to deform $D_n$ to obtain $D_n'$. It is obtained in the following way: $D_n$ can be seen as a subset of a genus zero translation surface $S_1$ in the stratum $\mathcal{H}(k_1,\dots ,k_r,-2-\sum_{i=1}^r k_i)$: we just ``glue'' a neighborhood of a pole to the boundary of the disk $D_n$. We proceed similarly with the disk $D_n'$ and obtain a translation surface $S_2$ in the same stratum as $S_1$. This stratum is connected since the genus is zero. Hence we deduce a continuous transformation that deform $D_n$ to $D_n'$.

From the last two paragraphs, we easily deduce a continuous path from $S_n$ to $S_n'$, which proves the proposition.
\end{proof}

The following proposition is the second step of the proof. It is the analogous of Proposition~5 and Proposition~6 in \cite{KoZo}. Our proof is also valid for the Abelian case, and gives an interesting alternative proof.

\begin{proposition}\label{join:cc}
\begin{enumerate}
\item Let $\mathcal{H}=\mathcal{H}(n,-p_1,\dots ,-p_s)$ be a genus $g\geq 2$  minimal stratum whose poles are all even or the pair $(-1,-1)$. 
For any $n_1,n_2$ odd such that $n_1+n_2=n$, there is a path $\gamma(t)\in \overline{\mathcal{H}(n_1,n_2,-p_1,\dots ,-p_s)}$ such that $\gamma(0),\gamma(1)\in \mathcal{H}$ and have different parities of spin structures.
\item Let $\mathcal{H}=\mathcal{H}(n,-p_1,\dots ,-p_s)$ be a genus $g\geq 2$  minimal stratum that contains a hyperelliptic connected component. For any $n_1\neq n_2$  such that $n_1+n_2=n$, there is a path $\gamma(t)\in \overline{\mathcal{H}(n_1,n_2,-p_1,\dots ,-p_s)}$ such that $\gamma(0)$ is in a hyperelliptic component of $\mathcal{H}$ and $\gamma(1)$ is in a nonhyperelliptic component of $\mathcal{H}$.
\end{enumerate}
\end{proposition}

\begin{proof}
\emph{Case (1)}\\
Let $\mathcal{C}_0=\mathcal{H}(n-2g,-p_1,\dots ,-p_s)$. The connected components of $\mathcal{H}$ given by $\mathcal{C}_0\oplus 1\dots \oplus 1 \oplus 1$ and $ \mathcal{C}_0\oplus 1\dots \oplus 1\oplus 2$ have different parities of spin structures. We can rewrite these components as
$ \mathcal{C}\oplus 1$ and $\mathcal{C}\oplus 2$, 
where $\mathcal{C}= \mathcal{C}_0\oplus 1\dots \oplus 1$.

Fix $S_{g-1}\in \mathcal{C}$. 
For a surface $S_1\in \mathcal{H}(n,-n)$, one can get a surface $S$ in $\mathcal{H}(n,-p_1,\dots ,-p_s)$ by the following surgery: 
\begin{itemize}
\item Cut $S_{g-1}$ along a small metric circle that turns around the singularity of degree $n-2$, and remove the disk bounded by this circle
\item Cut $S_1$ along a large circle that turns around the pole of order $n$, and rescale $S_1$ such that this circle is isometric to the previous one. Remove the neighborhood of the pole of order $n$ bounded by this circle.
\item Glue the two remaning surfaces along these circle, to obtain a surface $S\in \mathcal{H}(n,-p_1,\dots ,-p_s)$.
\end{itemize}
All choices in previous construction lead to the same connected component of $\mathcal{H}(n,-p_1,\dots ,-p_s)$, once $S_{g-1}, S_1$ are fixed. 
Similarly, we can do the same starting from a surface in $S_1\in\mathcal{H}(n_1,n_2,-n)$ and get a surface in $\mathcal{H}(n_1,n_2,-p_1,\dots ,-p_s)$. 

Now we start from a surface $S_{1,1} \in \mathcal{H}(n,-n)$ obtained by bubbling a handle with angle $2\pi$, \emph{i.e.} $S_{1,1}\in \mathcal{H}(n-2,-n)\oplus 1$. The rotation number of this surface is $\gcd(1,n)=1$. Breaking up the singularity into two singularities of order $n_1,n_2$,  the rotation number is still $1$.
Similarly, start from $S_{1,2}\in \mathcal{H}(n-2,-n)\oplus 2$. Its rotation number is $\gcd(2,n)=2$. Breaking up the singularity into two singularities of order $n_1,n_2$,  the rotation number becomes $\gcd(2,n_1,n_2)=1$ since $n_1,n_2$ are odd. Hence there is a path in $\overline{\mathcal{H}(n_1,n_2,-n)}$ that joins $S_{1,1}\in\mathcal{H}(n-2,-n)\oplus 1$  to  
$S_{1,2}\in \mathcal{H}(n-2,-n)\oplus 2$. From this path, we deduce a path in $\overline{\mathcal{H}(n_1,n_2,-p_1,\dots ,-p_s)}$ that joins $\mathcal{C}\oplus 1$ to $\mathcal{C}\oplus 2$. So Part~$(1)$ of the proposition is proven.
\medskip

\noindent
Case (2)\\
The proof is similar as the previous one: the hyperelliptic component of $\mathcal{H}(n,-p_1,\dots ,-p_s)$ is of the kind $\mathcal{C}\oplus \frac{n}{2}$, for some component $\mathcal{C}$.  Any component of the kind $\mathcal{C}\oplus k$, with $k\neq \frac{n}{2}$ is nonhyperelliptic. As before, we use the case of genus one strata.
A surface in $\mathcal{H}(n-2,-n)\oplus \frac{n}{2}$ is of rotation number $\gcd(\frac{n}{2},n)=\frac{n}{2}$. Breaking up the singularity of degree $n$ into two singularities of degree $n_1,n_2$, one obtains surface in $\mathcal{H}(n_1,n_2,-n)$ of rotation number $\gcd(\frac{n}{2},n_1,n_2)$. Since $n_1+n_2=n$ and $n_1\neq n_2$, this rotation number is not $\frac{n}{2}$, but some integer $k\in \{1,\dots ,\frac{n}{2}-1\}$. Hence there is a path in $\overline{\mathcal{H}(n_1,n_2,-n)}$ that joins  $\mathcal{H}(n-2,-n)\oplus \frac{n}{2}$ to  $\mathcal{H}(n-2,-n)\oplus k$. From this, we deduce the required path in $\overline{\mathcal{H}(n_1,n_2,-p_1,\dots ,-p_s)}$.
\end{proof}

%\begin{lemma}
%Let $\mathcal{H}(n_1,\dots ,n_r,-p_1,\dots ,-p_s)$ be a stratum of meromorphic differentials. We have the following:
%\begin{enumerate}
%\item If there exists $n_i$ odd, then the stratum is adjacent to a stratum of the kind $\mathcal{H}(m_1,m_2,-p_1,\dots ,-p_s)$, with $m_1$ odd.
%\item If $r>2$, then the stratum is adjacent to a stratum of the kind $\mathcal{H}(m_1,m_2,-p_1,\dots ,-p_s)$ with $m_1\neq m_2$.
%\end{enumerate}
%\end{lemma}
%\comment{Il manque adjacence des strates hyp/non hyp.}

Now we have all the intermediary results to prove  Theorem~\ref{MT2}.

\begin{proof}[Proof of Theorem~\ref{MT2}]
Let $\mathcal{H}=\mathcal{H}(n_1,\dots ,n_r,-p_1,\dots ,-p_s)$ be a stratum of genus $g\geq 2$ surfaces. Denote by $\mathcal{H}_{min}$ the minimal stratum obtained by collapsing all zeroes. Recall that by Proposition~\ref{prop:upper:bound}, the number of connected components of $\mathcal{H}$ is smaller than, or equal to the number of connected components of $\mathcal{H}_{min}$.

If $\sum_i p_i$ is odd, then the minimal stratum is connected and therefore the stratum is connected. So we can assume that $\sum_i p_i$ is even.

Assume that $\sum p_i>2$ or $g>2$.  From Theorem~\ref{th:str:min}, $\mathcal{H}_{min}$, hence $\mathcal{H}$ has at most three components.

We fix some vocabulary: we say that the set of degree of zeroes (resp. poles) is of \emph{hyperelliptic type} if this set is $\{n,n\}$ or $\{2n\}$ (resp. $\{-p,-p\}$ or $\{-2p\}$), \emph{i.e.} it is the set of degree of zeroes or poles of a hyperelliptic component. Note that the set of degree of poles are of hyperelliptic type if and only if the corresponding minimal stratum contains a hyperelliptic connected component. We will also say that the set of degree of poles is of \emph{even type} if the degrees of the zeroes are all even or if they are $\{-1,-1\}$. This means that the underlying minimal stratum has two nonhyperelliptic components  distinghished by the parity of the spin structure.

\begin{itemize}
\item If the stratum  is $\mathcal{H}(n,n,-2p)$ or $\mathcal{H}(n,n,-p,-p)$. There is a hyperelliptic connected component. The corresponding minimal stratum $\mathcal{H}(2n,*)$ has one hyperelliptic component and at least one nonhypereliptic component. It is easy to see that breaking up the singularity of degree $2n$ into two singularities of degree $n$, from a nonhyperelliptic translation surface gives a surface in a nonhyperelliptic connected component. So, the stratum $\mathcal{H}(n,n,*)$ has one hyperelliptic connected component and at least one nonhyperelliptic connected component.
\item If the set of degrees of poles and zeroes is of even type, we know from Theorem~\ref{th:str:min} that the minimal stratum has  two nonhyperelliptic components (and possibly one hyperelliptic). Breaking up the singularity into even degree singularities preserves the spin structure, which therefore gives at least two nonhyperelliptic components in the stratum.
\end{itemize}
From the above description, we obtain  lower bounds on the number of connected components. In particular, we see that if the degrees of zeroes and poles are both of hyperelliptic and even type, $\mathcal{H}$ as at least, so exactly, three connected components. Also, if the set of degrees of zeroes and poles is of hyperelliptic or even type, $\mathcal{H}$ has at least two connected components.

Now we give upper bounds.
\begin{enumerate}
\item Assume that the poles are of hyperelliptic and even type, \emph{i.e.} the minimal stratum has three connected components. Denote respectively by $\mathcal{C}^{hyp}, \mathcal{C}^{odd}$ and $\mathcal{C}^{even}$ the connected components of $\mathcal{H}$ that are adjacent respectively to the three connected components of $\mathcal{H}_{min}$, $\mathcal{H}_{min}^{hyp},\mathcal{H}_{min}^{odd}$ and $\mathcal{H}_{min}^{even}$. For any $j\in \{1,\dots ,r\}$, the stratum $\mathcal{H}(n_j,\sum_{i\neq j} n_i,-p_1,\dots ,-p_s)$ is adjacent to $\mathcal{H}_{min}$. 
\begin{itemize}
\item 
If the zeroes are not of hyperelliptic type, we can choose, $n_i$ so that $n_i\neq \sum_{i\neq j} n_i$, and by Proposition~\ref{join:cc}
  there is a path in $\mathcal{H}(n_j,\sum_{i\neq j} n_i,-p_1,\dots ,-p_s)$ joining the hyperelliptic component of $\mathcal{H}_{min}$ to a nonhyperelliptic connected component. Breaking up the singularity of order $\sum_{i\neq j}^r n_i$ along this path into singularities of order $(n_i)_{i\neq j}$, we obtain a path in $\mathcal{H}$ that joins a neighborhood of $\mathcal{H}_{min}^{hyp}$ to a neighborhood of a nonhyperelliptic component of $\mathcal{H}_{min}$. Hence, we necessarily have $\mathcal{C}^{hyp}=\mathcal{C}^{odd}$ or $\mathcal{C}^{hyp}=\mathcal{C}^{even}$.
 \item 
 If the zeroes are not even, we conclude similarly that $\mathcal{C}^{odd}=\mathcal{C}^{even}$
 \item Note that if the zeroes are neither of hyperelliptic type nor of even type, then $\mathcal{C}^{even}=\mathcal{C}^{odd}=\mathcal{C}^{hyp}$, so there is only one component for $\mathcal{H}$.
\end{itemize}

\item Assume that the poles are of hyperelliptic type but not of even type. The minimal stratum has two connected components, so there are at most two connected components for $\mathcal{H}$. 
If the zeroes are of hyperelliptic type, we have already seen that there are two components. 

Assume the zeroes are not of hyperelliptic type. 
Denote respectively by $\mathcal{C}^{hyp}, \mathcal{C}^{nonhyp}$ the connected components of $\mathcal{H}$ that are adjacent respectively to the hyperelliptic and the nonhyperelliptic component of $\mathcal{H}_{min}$. 
By the same argument as in $(1)$, using Proposition~\ref{join:cc} we have $\mathcal{C}^{hyp}=\mathcal{C}^{nonhyp}$, so $\mathcal{H}$ is connected.

\item Assume that the poles are of even type but not of hyperelliptic type. The minimal stratum has two connected components distinguished by the parity of the spin structure. So there are at most two components for $\mathcal{H}$. If the zeroes are of even type, there are exactly two connected component for $\mathcal{H}$, that are distinguished by the parity of the spin structure.

If the zeroes are not of even type, denote respectively by $\mathcal{C}^{odd}, \mathcal{C}^{even}$ the connected components of $\mathcal{H}$ that are adjacent respectively to the two components of $\mathcal{H}_{min}$. 
By the same argument as in $(1)$, using Proposition~\ref{join:cc} we have $\mathcal{C}^{odd}=\mathcal{C}^{even}$.

\item Assume that the poles are neither of hyperelliptic nor of even type, then the minimal stratum is connected, so $\mathcal{H}$ is connected.
\end{enumerate}

It remains to prove the theorem when $g=2$ and $\sum_{i} p_i=2$. The minimal stratum has two connected components. In this case, it is equivalent to say that the  zeroes are of hyperelliptic type or to say that they are of even type. If $\mathcal{H}=\mathcal{H}(2,2,*)$ or $\mathcal{H}(4,*)$, the stratum has at least two components, so exactly two. 
Otherwise, the stratum is adjacent to $\mathcal{H}(3,1,*)$, which connects $\mathcal{H}_{min}^{odd}$ to $\mathcal{H}_{min}^{even}$, hence $\mathcal{H}$ is connected.
\end{proof}

\appendix
\section{Negative results for meromorphic differentials}
In this section, we quickly give some examples to show that many well known results for the dynamics on translation surfaces are false in the case of translation surfaces with poles.
\subsection{Dynamics of the geodesic flow}
On a standard translation surface, the geodesic flow is uniquely ergodic for almost any directions. From the result of Proposition~\ref{no:vertical}, for almost any direction on a translation surface with poles, all infinite orbits for the geodesic flow converge to a pole.

\subsection{Cylinders and closed geodesics}
On a standard translation surface, always exists infinitely many closed geodesics (hence cylinders). For the case of translation surface with poles, one can consider the following example. Take the plane $\mathbb{C}$ and remove the inside of a square, and glue together by translation the corresponding opposite sides. One gets a surface in $\mathcal{H}(-2,2)$. It is easy to see that there are exactly two saddle connections joining the conical singularity to itself and no closed geodesic. A similar example in $\mathcal{H}(-2,1,1)$ obtained by removing a regular hexagon gives an example without a single saddle connection joining a conical singularity to itself.

\subsection{$SL_2(\mathbb{R})$ action}
 The $SL_2(\mathbb{R})$ action on the strata of the moduli space of Abelian differentials is ergodic. It is not the case for the moduli space of meromorphic differentials if we consider the (infinite) volume form defined by the flat local coordinates. Indeed, consider the stratum $\mathcal{H}(-2,2)$, which is connected. Consider the set of surfaces obtained with  the infinite zippered rectangle construction, by gluing together the set $D^+(z_1,z_2)$ and the set $D^-(z_2,z_1)$. It is easy to see that if $Im(z_2)<0<Im(z_1)$, there are no cylinders on the surface while if $Im(z_2)>0>Im(z_1)$, there is a cylinder on the surface. These two cases form two nointersecting open subsets of $\mathcal{H}(-2,2)$. Considering $SL_2(\mathbb{R})$ orbits, we obtain two disjoints $SL_2(\mathbb{R})$-invariants open subsets of a connected stratum.

\nocite*
\bibliographystyle{plain}
\bibliography{biblio}

\end{document}